\documentclass[12pt,reqno]{amsart}
\usepackage{amsmath}
\usepackage{amsthm}
\usepackage{amssymb}
\usepackage{amsfonts}
\usepackage{latexsym}
\usepackage{hyperref}
\usepackage{cite}
\usepackage{xcolor}
\usepackage{fullpage}

\usepackage{todonotes}
\usepackage{comment}
\usepackage{multirow}

\theoremstyle{plain}
\newtheorem{theorem}{Theorem}[section]
\newtheorem{corollary}[theorem]{Corollary}
\newtheorem{lemma}[theorem]{Lemma}

\theoremstyle{definition}

\newtheorem*{remark}{Remark}

\newcommand{\K}{\mathbb{K}}

\renewcommand{\L}{\mathbb{L}}

\newcommand{\Li}{\operatorname{Li}}
\newcommand{\Ei}{\operatorname{Ei}}
\newcommand{\Gal}{\operatorname{Gal}}

\usepackage{longtable}

\newcommand{\imod}[1]{\,\left(\textnormal{mod }#1\right)\,}

\begin{document}

\title[Conditional PNT4PAP]{The prime number theorem for primes in arithmetic progressions at large values}

\author[E.~S.~Lee]{Ethan Simpson Lee}
\address{University of Bristol, School of Mathematics, Fry Building, Woodland Road, Bristol, BS8 1UG} 
\email{ethan.lee@bristol.ac.uk}
\urladdr{\url{https://sites.google.com/view/ethansleemath/home}}

\thanks{ESL thanks Oleksiy Klurman, Aleksander Simoni\v{c}, and Tim Trudgian for valuable feedback, and the Heilbronn Institute for Mathematical Research for their support. He also thanks Dave Platt for evaluating some sums over zeros up to a small height (see \eqref{eqn:thank_Dave}).}


\hypersetup{
pdftitle={\@title},
pdfauthor={Ethan Simpson Lee},
pdfsubject={Number Theory},
pdfkeywords={prime number theorem;arithmetic progressions;GRH;Riemann hypothesis}
}

\maketitle

\begin{abstract}
    Assuming the Riemann hypothesis, we prove the latest explicit version of the prime number theorem for short intervals. Using this result, and assuming the generalised Riemann hypothesis for Dirichlet $L$-functions is true, we then establish explicit formulae for $\psi(x,\chi)$, $\theta(x,\chi)$, and an explicit version of the prime number theorem for primes in arithmetic progressions that hold for general moduli $q\geq 3$. Finally, we restrict our attention to $q\leq 10\,000$ and use an exact computation to refine these results.
\end{abstract}

\section{Introduction}

Suppose that $x\geq 2$, $p$ are prime numbers, $\chi$ is a Dirichlet character modulo $q\geq 3$, 
\begin{equation*}
    \psi(x,\chi) = \sum_{n\leq x} \chi(n)\Lambda(n),
    \quad\text{and}\quad
    \psi_1(x,\chi) 
    = \int_0^x \psi(t,\chi)\,dt
    = \sum_{n\leq x} \chi(n) \Lambda(n) (x-n).
\end{equation*}
The purpose of this paper is to prove the latest explicit and conditional version of the prime number theorem for primes in arithmetic progressions, which is a collection of asymptotic estimates for
\begin{equation*}
    \pi(x;q,a) 
    = \sum_{\substack{p \leq x \\ p \equiv a \imod q}} 1,
    \quad
    \theta(x;q,a) 
    = \sum_{\substack{p \leq x \\ p \equiv a \imod q}} \log{p},
    \quad\text{and}\quad
    \psi(x;q,a) 
    = \sum_{\substack{n \leq x \\ n \equiv a \imod q}} \Lambda(n) ,
\end{equation*}
in which $0\leq a < q$ is an integer such that $(a,q)=1$. Explicit bounds for each of these counting functions is a natural consequence of an explicit bound for $\psi(x,\chi)$, since
\begin{equation}\label{eqn:convenient}
    \psi(x;q,a) = \varphi(q)^{-1} \sum_{\chi}\overline{\chi}(a)\psi(x,\chi) ,
\end{equation}
where $\varphi(q)$ is the Euler-totient function evaluated at $q$;
we will see how a result for $\psi(x;q,a)$ naturally leads to results for $\theta(x;q,a)$ and $\pi(x;q,a)$ in Section \ref{sec:corollary2}. Effective versions of the prime number theorem for primes in arithmetic progressions are a key ingredient in many applications, such as effective versions of Chen's theorem; see \cite{Bordignon_Chen1, BordignonJohnstonStarichkova_Chen, BordignonStarichkova_ChenGRH}.

It is a straightforward application of \cite[(6.2)]{Schoenfeld}, to find effective bounds for $\psi(x,\chi_0)$; see \eqref{eqn:principal_chi_case} and recall that $\chi_0$ is standard notation for the \textit{principal} character. Therefore, the main challenge in this paper is to find good explicit bounds for $\psi(x,\chi)$ such that $\chi\neq\chi_0$, which we do in Section \ref{sec:twisted_psi}. Once this is done, we use these bounds to prove the following result which is our explicit version of the prime number theorem for primes in arithmetic progressions.

\begin{corollary}\label{cor:pntpaps}
If the GRH is true, $(a,q)=1$, $x\geq x_0 \geq \max\{e^{10},q\}$, and $q\geq 3$, then Table \ref{tab:pntpaps} presents constants $a_i$ such that
\begin{align}
    \left|\pi(x;q,a) - \frac{\Li(x)}{\varphi(q)}\right|
    &< \left(\frac{\log{x}}{8\pi} + \frac{a_1 \log{q}}{2\pi} + a_2\right) \sqrt{x} + a_3 , \label{eqn:pi_xqa_result}\\
    \left|\theta(x;q,a) - \frac{x}{\varphi(q)}\right|
    &< \left(\frac{\log{x}}{8\pi} + \frac{\log{q}}{2\pi} + a_4\right) \sqrt{x}\log{x} + a_5 , \nonumber\\
    \left|\psi(x;q,a) - \frac{x}{\varphi(q)}\right|
    &< \left(\frac{\log{x}}{8\pi} + \frac{\log{q}}{2\pi} + a_6\right) \sqrt{x}\log{x} + a_5 . \nonumber
\end{align}
\end{corollary}

\begin{table}[]
    \centering
    \begin{tabular}{c|cccccc}
        $\log{x_0}$ & $a_1$ & $a_2$ & $a_3$ & $a_4$ & $a_5$ & $a_6$ \\
        \hline
        $10$ & $1.27146$ & $11.85396$ & $-1.19899\cdot 10^{4}$ & $9.29179$ & $-8.31179\cdot 10^{3}$ & $7.84909$ \\
        $20$ & $1.11315$ & $7.0938$ & $-4.55028\cdot 10^{6}$ & $6.33697$ & $-3.15402\cdot 10^{6}$ & $4.89427$ \\
        $30$ & $1.07187$ & $6.11552$ & $-1.16988\cdot 10^{9}$ & $5.66833$ & $-8.10901\cdot 10^{8}$ & $4.22563$ \\
        $40$ & $1.0528$ & $5.63974$ & $-2.37009\cdot 10^{11}$ & $5.31911$ & $-1.64282\cdot 10^{11}$ & $3.87641$ \\
        $50$ & $1.04175$ & $5.36916$ & $-4.47689\cdot 10^{13}$ & $5.11581$ & $-3.10314\cdot 10^{13}$ & $3.67311$ \\
        $60$ & $1.03453$ & $5.18036$ & $-8.06659\cdot 10^{15}$ & $4.96901$ & $-5.59133\cdot 10^{15}$ & $3.52631$ \\
        $70$ & $1.02944$ & $5.04344$ & $-1.4063\cdot 10^{18}$ & $4.86057$ & $-9.74775\cdot 10^{17}$ & $3.41787$ \\
        $80$ & $1.02566$ & $4.93893$ & $-2.36519\cdot 10^{20}$ & $4.77658$ & $-1.63943\cdot 10^{20}$ & $3.33388$ \\
        $90$ & $1.02274$ & $4.84652$ & $-3.99108\cdot 10^{22}$ & $4.69986$ & $-2.7664\cdot 10^{22}$ & $3.25716$ \\
        $100$ & $1.02042$ & $4.23696$ & $-6.61986\cdot 10^{24}$ & $4.11319$ & $-4.58854\cdot 10^{24}$ & $2.67049$ \\
        $150$ & $1.01352$ & $1.58052$ & $-7.2138\cdot 10^{35}$ & $1.52019$ & $-5.00022\cdot 10^{35}$ & $0.07749$ \\
        $200$ & $1.0101$ & $0.15187$ & $-6.88948\cdot 10^{46}$ & $0.11096$ & $-4.77543\cdot 10^{46}$ & $-1.33174$ \\
        $250$ & $1.00807$ & $-1.49591$ & $-6.36488\cdot 10^{57}$ & $-1.52341$ & $-4.4118\cdot 10^{57}$ & $-2.96611$ \\
        $500$ & $1.00402$ & $-10.73303$ & $-2.39569\cdot 10^{112}$ & $-10.72973$ & $-1.66057\cdot 10^{112}$ & $-12.17243$
    \end{tabular}
    \caption{Admissible constants for $a_i$ at several choices of $x_0$ in Corollary \ref{cor:pntpaps}.}
    \label{tab:pntpaps}
\end{table}

If $q\leq 10\,000$, then we can also refine the constants in Corollary \ref{cor:pntpaps}. Our refined result, which is presented below, follows from refinements to the key ingredients in the proof of Corollary \ref{cor:pntpaps}; most of those improvements are a consequence of the exact computation presented in \eqref{eqn:thank_Dave}. It is also important to note that we were able to remove the condition $x_0\geq q$ in this refinement.

\begin{corollary}\label{cor:pntpaps_small_moduli}
If the GRH is true, $(a,q)=1$, $x\geq x_0\geq 1.05\cdot 10^7$, and $3\leq q\leq 10\,000$, then Table \ref{tab:pntpaps_small_moduli} presents constants $\widetilde{a_i}$ such that
\begin{align}
    \left|\pi(x;q,a) - \frac{\Li(x)}{\varphi(q)}\right|
    &< \left(\frac{\log{x}}{8\pi} + \frac{\widetilde{a_1} \log{q}}{2\pi} + \widetilde{a_2}\right) \sqrt{x} + \widetilde{a_3} , \nonumber\\
    \left|\theta(x;q,a) - \frac{x}{\varphi(q)}\right|
    &< \left(\frac{\log{x}}{8\pi} + \frac{\log{q}}{2\pi} + \widetilde{a_4}\right) \sqrt{x}\log{x} + \widetilde{a_5} , \nonumber\\
    \left|\psi(x;q,a) - \frac{x}{\varphi(q)}\right|
    &< \left(\frac{\log{x}}{8\pi} + \frac{\log{q}}{2\pi} + \widetilde{a_6}\right) \sqrt{x}\log{x} + \widetilde{a_5} . \nonumber
\end{align}
Note that $a_1 = \widetilde{a_1}$ and we present values for $\log{x_0} \geq 20$ in Table \ref{tab:pntpaps_small_moduli}, because $20$ is the smallest multiple of ten greater than $\log{1.05\cdot 10^7}$.
\end{corollary}

\begin{table}[]
    \centering
    \begin{tabular}{c|cccccc}
        $\log{x_0}$ & $\widetilde{a_1}$ & $\widetilde{a_2}$ & $\widetilde{a_3}$ & $\widetilde{a_4}$ & $\widetilde{a_5}$ & $\widetilde{a_6}$ \\
        \hline
        $20$ & $1.11315$ & $-10.80603$ & $-4.55029\cdot 10^{6}$ & $-9.74334$ & $-3.15402\cdot 10^{6}$ & $-11.18604$ \\
        $30$ & $1.07187$ & $-10.63523$ & $-1.16988\cdot 10^{9}$ & $-9.95921$ & $-8.10901\cdot 10^{8}$ & $-11.40191$ \\
        $40$ & $1.0528$ & $-10.57626$ & $-2.37009\cdot 10^{11}$ & $-10.08365$ & $-1.64282\cdot 10^{11}$ & $-11.52635$ \\
        $50$ & $1.04175$ & $-10.53537$ & $-4.47689\cdot 10^{13}$ & $-10.15137$ & $-3.10314\cdot 10^{13}$ & $-11.59407$ \\
        $60$ & $1.03453$ & $-10.52003$ & $-8.06659\cdot 10^{15}$ & $-10.20738$ & $-5.59133\cdot 10^{15}$ & $-11.65008$ \\
        $70$ & $1.02944$ & $-10.51273$ & $-1.4063\cdot 10^{18}$ & $-10.25074$ & $-9.74775\cdot 10^{17}$ & $-11.69344$ \\
        $80$ & $1.02566$ & $-10.50982$ & $-2.36519\cdot 10^{20}$ & $-10.28569$ & $-1.63943\cdot 10^{20}$ & $-11.72839$ \\
        $90$ & $1.02274$ & $-10.50942$ & $-3.99108\cdot 10^{22}$ & $-10.31466$ & $-2.7664\cdot 10^{22}$ & $-11.75736$ \\
        $100$ & $1.02042$ & $-10.51054$ & $-6.61986\cdot 10^{24}$ & $-10.33923$ & $-4.58854\cdot 10^{24}$ & $-11.78193$ \\
        $150$ & $1.01352$ & $-10.52454$ & $-7.2138\cdot 10^{35}$ & $-10.42345$ & $-5.00022\cdot 10^{35}$ & $-11.86615$ \\
        $200$ & $1.0101$ & $-10.54074$ & $-6.88948\cdot 10^{46}$ & $-10.47471$ & $-4.77543\cdot 10^{46}$ & $-11.91741$ \\
        $250$ & $1.00807$ & $-10.55546$ & $-6.36488\cdot 10^{57}$ & $-10.51048$ & $-4.4118\cdot 10^{57}$ & $-11.95318$ \\
        $500$ & $1.00402$ & $-10.60614$ & $-2.39569\cdot 10^{112}$ & $-10.60334$ & $-1.66057\cdot 10^{112}$ & $-12.04604$
    \end{tabular}
    \caption{Admissible constants for $\widetilde{a_i}$ at several choices of $x_0$ in Corollary \ref{cor:pntpaps_small_moduli}.}
    \label{tab:pntpaps_small_moduli}
\end{table}

The author's motivation for producing these results was to establish explicit versions of the prime number theorem for primes in arithmetic progressions at larger values of $x$, to explore how large $x$ needs to be before the (following) expected bound becomes true:
\begin{equation*}
    \left|\pi(x;q,a) - \frac{\Li(x)}{\varphi(q)} \right|
    < \left(\frac{\log{x}}{8\pi} + \frac{\log{q}}{2\pi}\right) \sqrt{x}.
\end{equation*}
Indeed, Greni\'{e} and Molteni already give a good explicit version of \eqref{eqn:pi_xqa_result} that holds for $x\geq 2$ as a special case of their results in \cite{Grenie2019}; see \eqref{eqn:GM_CDT_sc} in Appendix \ref{sec:cdt}. Recently, Ernvall-Hyt\"{o}nen and Paloj\"{a}rvi also proved another explicit version of the prime number theorem for primes in arithmetic progressions in \cite{ErnvallHytonenPalojarvi}, although \eqref{eqn:GM_CDT_sc} appears to be sharper. Moreover, the leading terms in \eqref{eqn:pi_xqa_result} are marginally better than the leading terms in \eqref{eqn:GM_CDT_sc} and our major improvements are to the secondary terms.

Now, Corollaries \ref{cor:pntpaps} and \ref{cor:pntpaps_small_moduli} are consequences of the explicit bounds for $\psi(x,\chi)$ that are presented in Theorems \ref{thm:psi_chi_expl_fmla} and \ref{thm:psi_chi_expl_fmla_small_moduli} respectively. We will prove these in Section \ref{sec:twisted_psi}, but offer a brief overview of the approach here first, because our approach contains novelty. The main challenge is to prove the result for \textit{primitive} non-principal characters, because it is a straightforward task to extend that observation to \textit{any} non-principal character, which we do in the final steps. Now, to prove the result for these \textit{primitive} characters, we note that $\psi(x,\chi)$ differs from
\begin{equation*}
    \Delta(x,\chi) = \frac{\psi_1(x+\sqrt{x}\log{x},\chi) - \psi_1(x,\chi)}{\sqrt{x}\log{x}}
\end{equation*}  
by a small error of size $O(\sqrt{x}\log{x})$; this is a consequence of the prime number theorem for short intervals, which requires $\log{x}\geq 10$. To this end, we will prove the latest explicit prime number theorem for short intervals in Section \ref{sec:PNT4SI}; see Theorem \ref{thm:CH_D_refined}. This result refines recent work by Cully-Hugill and Dudek \cite{CullyHugillDudek} and extends their computations, by splitting a particular sum over the non-trivial zeros of the Riemann zeta-function $\zeta(s)$ into more sub-intervals. All that remains is to bound $\Delta(x,\chi)$ using an explicit formula for $\psi_1(x,\chi)$; this will establish the relationship  
\begin{equation}\label{eqn:loose_expl_form}
     \Delta(x,\chi) \sqrt{x}\log{x} \approx - \sum_{\varrho_{\chi}} \frac{(x+h)^{\varrho_{\chi}+1} - x^{\varrho_{\chi}+1}}{\varrho_{\chi} (\varrho_{\chi}+1)},
\end{equation}
in which $\varrho_{\chi} = 1/2 + i\gamma_{\chi}$ are the non-trivial zeros of $L(s,\chi)$ and the error will be completely described. Finally, we apply bounds for the absolute value of this sum over zeros, which we obtain in Section \ref{sec:important_soz}. The novel ideas in this approach are that we bound the difference between $\psi(x,\chi)$ and $\Delta(x,\chi)$ using the prime number theorem for short intervals and the sub-intervals that we split the sums over zeros that arise into.

\section{Tools and results for sums over zeros}\label{sec:important_soz}

Suppose $\rho = \beta + i\gamma$ are the non-trivial zeros of the Riemann zeta-function $\zeta(s)$, $N(T)$ is the number of $\rho$ such that $0<\gamma\leq T$, and $\phi: [T_0, \infty) \to [0,\infty)$ is a monotone, non-increasing function on $[T_0, \infty)$ for some $T_0 \geq 1$. To ensure $\phi(t)$ is convex on $[T_0,\infty)$, suppose further that $\phi$ is twice continuously differentiable, $\phi'(t) \leq 0$, and $\phi''(t) \geq 0$. Sums of the following form appear often in number theory (and multiple times in the present paper):
\begin{equation*}
    \sum_{\gamma \in [T_1, T_2]}' \phi(\gamma) .
\end{equation*}
Here, the prime symbol indicates that if $\gamma \in \{T_1, T_2\}$, then $\phi(\gamma)$ is weighted by $1/2$. Sums of this type can be estimated using a result \cite[Lem.~3]{BrentAccurate} from Brent, Platt, and Trudgian, which refines a technique used to prove \cite[Lem.~1]{Lehman}.


\begin{lemma}[Brent, Platt, Trudgian]\label{lem:BPT}
Let $A_0 = 2.067$, $A_1 = 0.059$, and $A_2 = 1/150$. If $2\pi \leq U \leq V$, then
\begin{equation*}
    \sum_{\substack{U \leq \gamma \leq V}}' \phi(\gamma)
    = \frac{1}{2\pi} \int_{U}^{V} \phi(t)\log{\frac{t}{2\pi}}\,dt + \phi(V)Q(V) - \phi(U)Q(U) + E_2(U,V),
\end{equation*}
in which $|E_2(U,V)| \leq 2(A_0 + A_1\log{U})|\phi'(U)| + (A_1 + A_2) \frac{\phi(U)}{U}$ and \cite[Cor.~1]{BPT_mean_square}, \cite[Cor.~1.2]{HasanalizadeShenWong} tell us
\begin{align}
    |Q(T)| 
    &= \left|N(T) - \frac{T}{2\pi} \log{\frac{T}{2\pi e}} - \frac{7}{8}\right| \nonumber\\
    &\leq R(T) := \min\{0.28\log{T}, 0.1038\log{T} + 0.2573\log\log{T} + 9.3675\} . \label{eqn:me_again}
\end{align}
\end{lemma}

To prove Lemma \ref{lem:BPT}, the authors rewrite the sum over zeros as a Stieltjes integral, apply integration by parts twice, and apply explicit versions of
\begin{align*}
    Q(T) &\ll \log{T},\\
    S(T) &= \pi^{-1}\arg{\zeta(\tfrac{1}{2} + iT)} \ll \log{T},
    \quad\text{and}\quad
    S_1(T) = \int_0^T S(t)\,dt \ll \log{T}.
\end{align*}
To prove a result of the shape of \cite[Lem.~1]{Lehman}, one should only apply integration by parts once, then apply an explicit version of $Q(T) \ll \log{T}$; the other ingredients are \textit{not} required in this less developed technique. Both of these results will generalise naturally into other settings, as long we have the prerequisite ingredients that we have mentioned. In some settings, we might not have access to good off-the-shelf bounds for the appropriate analogues of $S(T)$ and $S_1(T)$, although decent bounds for the appropriate analogue of $Q(T)$ are more common. 

In Section \ref{sec:PNT4SI}, we will use Lemma \ref{lem:BPT} to refine of a special case of \cite[Thm.~1]{CullyHugillDudek}, which will be an important ingredient in our proof of Theorem \ref{thm:exact_formula}. We will also need to estimate sums over the non-trivial zeros $\varrho_{\chi} = 1/2 + i\gamma_{\chi}$ of Dirichlet $L$-functions $L(s,\chi)$ in our eventual proof of Theorem \ref{thm:exact_formula}, where $\chi$ is a primitive and non-principal character modulo $q\geq 3$. To this end, we import the following lemma, which is \cite[Cor.~1.2]{Bennett}, and an explicit estimate for the number $N(T,\chi)$ of non-trivial zeros of $L(s,\chi)$ such that $|\gamma_{\chi}|\leq T$.


\begin{lemma}\label{lem:NZeros2}
If $\chi$ is a character with conductor $q \geq 2$ and $T \geq 5/7$, then
\begin{equation*}
    \left|N(T,\chi)-\frac{T}{\pi}\log{\frac{qT}{2\pi e}}\right| 
    \leq 0.247\log{\frac{qT}{2\pi}} + 6.894.
\end{equation*}
\end{lemma}

Using Lemma \ref{lem:NZeros2}, we obtain the following generalisation of \cite[Lem.~1]{Lehman}.

\begin{lemma}\label{lem:Lehman_gen}
Let $\gamma_{\chi}$ denote the ordinates of the non-trivial (and non-exceptional) zeros of $L(s,\chi)$ and $\phi: [T_0, \infty) \to [0,\infty)$ be a monotone, non-increasing function on $[T_0, \infty)$ for some $T_0 \geq 5/7$ that is continuously differentiable such that $\phi'(t) \leq 0$. If $5/7 \leq U \leq V$, then
\begin{equation}\label{eqn:Lehman_gen}
    \sum_{\substack{U \leq |\gamma_{\chi}| \leq V}} \phi(\gamma_{\chi})
    = \frac{\log{q}}{\pi} \int_U^V \phi(t)\,dt + \frac{1}{\pi} \int_U^V \phi(t) \log{\frac{t}{2\pi}} \,dt + E_3(U,V),
\end{equation}
in which
\begin{align*}
    |E_3(U,V)| \leq 2 \phi(U) \left(0.247\log{\frac{qU}{2\pi}} + 6.894\right) + 0.247 \int_U^V \frac{\phi(t)}{t}\,dt .
\end{align*}
\end{lemma}

\begin{proof}
Without loss of generality, we can suppose $|\gamma_{\chi}|\not\in\{U,V\}$, since the result will extend naturally to these fringe cases by taking limits. Lemma \ref{lem:NZeros2} implies
\begin{align*}
    \sum_{\substack{U \leq |\gamma_{\chi}| \leq V}} \phi(\gamma_{\chi})
    &= \int_U^V \phi(t)\,d N(t,\chi) \\
    &= \frac{\log{q}}{\pi} \int_U^V \phi(t)\,dt + \frac{1}{\pi} \int_U^V \phi(t) \log\left(\frac{t}{2\pi}\right) \,dt + \int_U^V \phi(t)\,dQ(t,\chi),
\end{align*}
in which $|Q(t,\chi)| \leq 0.247\log{\frac{qT}{2\pi}} + 6.894$. Therefore, use integration by parts twice to see
\begin{align*}
    |E_3(U,V)| 
    &= \left|\phi(V) Q(V,\chi) - \phi(U) Q(U,\chi) - \int_U^V \phi'(t) Q(t,\chi)\,dt\right| \\
    &\leq \sum_{m\in\{U,V\}} \phi(m) \left(0.247\log{\frac{q m}{2\pi}} + 6.894\right) - \int_U^V \phi'(t) \left(0.247\log{\frac{qt}{2\pi}} + 6.894\right)\,dt \\
    &= 2 \phi(U) \left(0.247\log{\frac{qU}{2\pi}} + 6.894\right) + 0.247 \int_U^V \frac{\phi(t)}{t}\,dt . \qedhere
\end{align*}
\end{proof}

Using Lemma \ref{lem:Lehman_gen}, we can prove the following result, which will be important later.

\begin{theorem}\label{thm:soz_result}
If the GRH is true, $h=\sqrt{x}\log{x}$, and $x\geq x_0$ then there exist constants $k_1(x_0)$ and $k_2(x_0)$ such that
\begin{align*}
    \left|\sum_{\varrho_{\chi}} \frac{(x+h)^{\varrho_{\chi}+ 1} - x^{\varrho_{\chi}+ 1}}{h \varrho_{\chi}(\varrho_{\chi}+1)} \right| 
    \leq \left(\frac{\log{x}}{8\pi} + \frac{\log{q}}{2\pi} + k_1(x_0)\right)\sqrt{x}\log{x} + k_2(x_0)\sqrt{x}\log{q}.
\end{align*}
We present a selection of admissible $k_i(x_0)$ for several values of $x_0\geq e^{10}$ in Table \ref{tab:soz_results}.
\end{theorem}

\begin{table}[]
    \centering
    \begin{tabular}{c|cccc}
        $\log{x_0}$ & $k_1(x_0)$ & $\widetilde{k_1}(x_0)$ & $k_2(x_0)$ & $\widetilde{k_2}(x_0)$ \\
        \hline
        $10$ & $3.10557$ & - & $2.36179$ & - \\
        $\log(1.05\cdot 10^7)$ & $1.72106$ & $1.36974$ & $2.23552$ & $-1.03762$ \\
        $20$ & $1.39025$ & $1.10685$ & $2.2275$ & $-1.04564$ \\
        $30$ & $0.94152$ & $0.75267$ & $2.22574$ & $-1.0474$ \\
        $40$ & $0.71678$ & $0.57514$ & $2.22572$ & $-1.04742$ \\
        $50$ & $0.5812$ & $0.46789$ & $2.22572$ & $-1.04742$ \\
        $60$ & $0.49041$ & $0.39599$ & $2.22572$ & $-1.04742$ \\
        $70$ & $0.42533$ & $0.34439$ & $2.22572$ & $-1.04742$ \\
        $80$ & $0.37637$ & $0.30555$ & $2.22565$ & $-1.04742$ \\
        $90$ & $0.3383$ & $0.27523$ & $2.21596$ & $-1.04746$ \\
        $100$ & $0.31171$ & $0.25071$ & $1.68045$ & $-1.05312$ \\
        $150$ & $-0.39649$ & $-0.42406$ & $-6.09189$ & $-7.60473$ \\
        $200$ & $-1.72239$ & $-1.74308$ & $-13.9581$ & $-15.47094$ \\
        $250$ & $-3.30168$ & $-3.31822$ & $-21.84477$ & $-23.35762$ \\
        $500$ & $-12.37619$ & $-12.38446$ & $-61.41291$ & $-62.92575$
    \end{tabular}
    \caption{Admissible values for $k_i(x_0)$ and $\widetilde{k_i}(x_0)$ at several choices of $x_0$ in Theorems \ref{thm:soz_result}-\ref{thm:soz_result_small_moduli}.}
    \label{tab:soz_results}
\end{table}

\begin{proof}
To estimate the objective sum over $\varrho_{\chi}$, observe that
\begin{equation*}
    \left|\sum_{|\gamma_{\chi}|\geq T} \frac{(x+h)^{\varrho_{\chi} +1} - x^{\varrho_{\chi} +1}}{\varrho_{\chi}(\varrho_{\chi}+1) h}\right| \leq \bigg(\left(1+\frac{h}{x}\right)^{\frac{3}{2}} + 1\bigg) \frac{x^{\frac{3}{2}}}{h} \sum_{|\gamma_{\chi}| \geq T} \frac{1}{\gamma_{\chi}^2} 
\end{equation*}
and
\begin{align*}
    \left|\sum_{|\gamma_{\chi}| < T} \frac{(x+h)^{\varrho_{\chi}+ 1} - x^{\varrho_{\chi}+ 1}}{h \varrho_{\chi}(\varrho_{\chi}+1)}\right|
    = \left|\sum_{|\gamma_{\chi}| < T} \frac{1}{h \varrho_{\chi}} \int_{x}^{x+h} t^{\varrho_{\chi}}\,dt\right|
    &\leq \sqrt{x} \left(1+\frac{h}{x}\right)^{\frac{1}{2}} \sum_{|\gamma_{\chi}| < T} \frac{1}{\sqrt{\tfrac{1}{4} + \gamma_{\chi}^2}} .
\end{align*}
We will insert the choices $h = \sqrt{x}\log{x}$ and $T = \sqrt{x}/\log{x} := \eta(x) \geq \eta(x_0) > 5/7$ (so that $h=x/T$) into \eqref{eqn:GMStep3}, which is a combination of the preceding observations:
\begin{equation}\label{eqn:GMStep3}
\begin{split}
    &\left|\sum_{\varrho_{\chi}} \frac{(x+h)^{\varrho_{\chi}+ 1} - x^{\varrho_{\chi}+ 1}}{h \varrho_{\chi}(\varrho_{\chi}+1)} \right| 
    \leq \bigg(\left(1+\frac{h}{x}\right)^{\frac{3}{2}} + 1\bigg) \frac{x^{\frac{3}{2}}}{h} \sum_{|\gamma_{\chi}| \geq T} \frac{1}{\gamma_{\chi}^2} \\ 
    &\qquad\qquad\quad+ \sqrt{x} \left(1+\frac{h}{x}\right)^{\frac{1}{2}} \left(2 N(\tfrac{5}{7},\chi) + \sum_{\frac{5}{7} \leq |\gamma_{\chi}| < \eta(x_0)} \frac{1}{\sqrt{\tfrac{1}{4} + \gamma_{\chi}^2}} + \sum_{\eta(x_0) \leq |\gamma_{\chi}| < T} \frac{1}{|\gamma_{\chi}|}\right) .
\end{split}
\end{equation}
First, apply Lemma \ref{lem:NZeros2} to see 
\begin{align}
    2 N(\tfrac{5}{7},\chi)
    &\leq \frac{10}{7\pi}\log{\frac{5q}{14\pi e}} 
    + 0.494\log{\frac{5q}{14\pi}} + 13.788
    < 0.94873 \log{q} + 11.27041 \label{eqn:soz_useful_zd}
\end{align}
Next, apply Lemma \ref{lem:Lehman_gen} with $\phi(t) = (1/4 + t^2)^{-1}$ to see
\begin{align}
    \sum_{\frac{5}{7} \leq |\gamma_{\chi}| < \eta(x_0)} \frac{1}{\sqrt{\tfrac{1}{4} + \gamma_{\chi}^2}}  
    &\leq \left(\frac{0.494}{\sqrt{\tfrac{1}{4} + \frac{25}{49}}} + \int_{\frac{5}{7}}^{\eta(x_0)} \frac{\pi^{-1}\,dt}{\sqrt{\tfrac{1}{4} + t^2}} \right) \log{q} 
    + \int_{\frac{5}{7}}^{\eta(x_0)} \frac{\pi^{-1} \log\frac{t}{2\pi}}{\sqrt{\frac{1}{4} + t^2}} \,dt \nonumber\\
    &\qquad\qquad + \frac{2\left(0.247\log{\frac{5}{14\pi}} + 6.894\right)}{\sqrt{\tfrac{1}{4} + \frac{25}{49}}}  + 0.247 \int_{\frac{5}{7}}^{\eta(x_0)} \frac{t^{-1}\,dt}{\sqrt{\tfrac{1}{4} + t^2}} \nonumber\\
    &:= \nu_1(x_0) \log{q} + \nu_2(x_0) \label{eqn:soz_useful0}
\end{align}
Similarly, apply Lemma \ref{lem:Lehman_gen} with $\phi(t) = t^{-1}$ to see
\begin{align}
    &\sum_{\eta(x_0) \leq |\gamma_{\chi}| < T} \frac{1}{|\gamma_{\chi}|} 
    \leq \left(\frac{\log{T}}{\pi} + \frac{0.494}{\eta(x_0)} - \frac{\log{\eta(x_0)}}{\pi} \right) \log{q} 
    + \frac{(\log{\frac{T}{2\pi}})^2 - (\log{\frac{\eta(x_0)}{2\pi}})^2}{2\pi} \nonumber\\
    &\qquad\qquad\qquad\qquad\qquad\qquad\qquad + \frac{2\left(6.894 - 0.247\log{2\pi\eta(x_0)}\right)}{\eta(x_0)} + \frac{0.247}{\eta(x_0)} - \frac{0.247}{T} \nonumber\\
    &\qquad := \left(\frac{\log{T}}{\pi} + \nu_3(x_0)\right)\log{q} + \frac{(\log{T})^2}{2\pi} - 0.585 \log{T} + \nu_4(x_0) - \frac{0.247}{T} .\label{eqn:soz_useful1} 
\end{align}
Finally, apply Lemma \ref{lem:Lehman_gen} with $\phi(t) = t^{-2}$ to see
\begin{equation}\label{eqn:soz_useful2}
    \sum_{|\gamma_{\chi}|\geq T} \frac{T}{\gamma_{\chi}^2}
    \leq \left(\frac{1}{\pi} + \frac{0.494}{T}\right)\log{q} + \left(\frac{1}{\pi} + \frac{0.494}{T}\right)\log{T} - 0.2667 + \frac{13.0034}{T} . 
\end{equation}
So, we can apply \eqref{eqn:soz_useful_zd}, \eqref{eqn:soz_useful0}, \eqref{eqn:soz_useful1}, and \eqref{eqn:soz_useful2} in \eqref{eqn:GMStep3} to see that if $x_0 \geq e^{10}$, then
\begin{align}
    &\left|\sum_{\varrho_{\chi}} \frac{(x+h)^{\varrho_{\chi}+ 1} - x^{\varrho_{\chi}+ 1}}{h \varrho_{\chi}(\varrho_{\chi}+1)} \right| \nonumber\\
    &\quad\leq \sqrt{x} \left(\left(1+\frac{\log{x}}{\sqrt{x}}\right)^{\frac{1}{2}} \left(\frac{(\log{x})^2}{8\pi} + \frac{(\log\log{x} - \log{x} - 2\log{q})\log\log{x}}{2\pi} + \frac{\log{x}\log{q}}{2\pi} \right.\right. \nonumber\\
    &\quad\qquad\qquad\left.\left. + (\nu_1(x_0) + \nu_3(x_0) + 0.94873)\log{q} + \nu_2(x_0) + \nu_4(x_0) + 11.27041 \right.\right. \nonumber\\
    &\quad\qquad\qquad\qquad\qquad\qquad\qquad\qquad\qquad\qquad \left.\left.   - 0.2925\log{x} + 0.5850 \log\log{x} - \frac{0.247\log{x}}{\sqrt{x}}\right)\right. \nonumber\\
    &\quad\qquad\qquad\left. + \bigg(\left(1+\frac{\log{x}}{\sqrt{x}}\right)^{\frac{3}{2}} + 1\bigg) \left( \left(\frac{1}{\pi} + \frac{0.494 \log{x}}{\sqrt{x}}\right) \left(\frac{2\log{q} + \log{x} - 2\log\log{x}}{2} \right) \right.\right. \nonumber\\
    &\quad\qquad\qquad\qquad\qquad\qquad\qquad\qquad\qquad\qquad \left.\left. - 0.2667 + \frac{13.0034\log{x}}{\sqrt{x}}\right) \right) \nonumber\\
    &\quad\leq \left(f_1(x_0)\log{x} + f_2(x_0)\log{q} + f_3(x_0)\right)\sqrt{x}\log{x} 
    + f_4(x_0) \sqrt{x}\log{q} 
    + f_5(x_0) \sqrt{x} , \nonumber 
\end{align}
where
\begin{align*}
    f_1(x)
    &= \left(1+\frac{\log{x}}{\sqrt{x}}\right)^{\frac{1}{2}} \left( \frac{1}{8\pi} - \frac{\log\log{x}}{2\pi\log{x}} + \frac{1}{2\pi}\left(\frac{\log\log{x}}{\log{x}}\right)^2 \right) , \\
    f_2(x) 
    &= \left(1+\frac{\log{x}}{\sqrt{x}}\right)^{\frac{1}{2}} \left(\frac{1}{2\pi} - \frac{\log\log{x}}{\pi\log{x}}\right) , \\
    f_3(x)
    &= \left(1+\frac{\log{x}}{\sqrt{x}}\right)^{\frac{1}{2}} \left(\frac{0.5850\log\log{x}}{\log{x}} - 0.2925\right) \\
    &\qquad\qquad+ \bigg(\left(1+\frac{\log{x}}{\sqrt{x}}\right)^{\frac{3}{2}} + 1\bigg) \left(\frac{1}{2\pi} + \frac{0.247 \log{x} + 13.0034}{\sqrt{x}}\right) , \\
    f_4(x)
    &= \bigg(\left(1+\frac{\log{x}}{\sqrt{x}}\right)^{\frac{3}{2}} + 1\bigg) \left(\frac{1}{\pi} + \frac{0.494 \log{x}}{\sqrt{x}}\right) + \nu_1(x) + \nu_3(x) + 0.98473
    , \\
    f_5(x)
    &= (\nu_2(x) + \nu_4(x) + 11.27041) \left(1+\frac{\log{x}}{\sqrt{x}}\right)^{\frac{1}{2}} - 0.5334 .
\end{align*}
Now, $f_1(x)$ and $f_2(x)$ \textit{increase} in the ranges of $x$ we are considering and their limits are $1/(8\pi)$ and $1/(2\pi)$ respectively, so we have 
\begin{equation*}
    f_1(x) \leq \frac{1}{8\pi} 
    \qquad\text{and}\qquad
    f_2(x) \leq \frac{1}{2\pi} .
\end{equation*}
Moreover, the result follows with $k_1(x_0) = f_3(x_0) + \tfrac{f_5(x_0)}{\log{x_0}}$ and $k_2(x_0) = f_4(x_0)$.
\end{proof}

If $3\leq q\leq 10\,000$, the GRH is true, and $\varrho_{\chi} = 1/2 + i\gamma_{\chi}$ are the non-trivial zeros of $L(s,\chi)$, then computations show that
\begin{equation}\label{eqn:thank_Dave}
    \max_{\textrm{primitive }\chi\imod{q}} \sum_{|\gamma_{\chi}| \leq 200} \frac{1}{\sqrt{\tfrac{1}{4} + \gamma_{\chi}^2}} = \omega:= 21.664472\dots ,
\end{equation}
where this maximum is achieved at $q=9\,857$ (which is prime). A sensible hypothesis is that the $L$-functions with \textit{most} low-lying zeros will maximise the sum in question. Now, using \eqref{eqn:thank_Dave}, we can refine Theorem \ref{thm:soz_result} in the following result.

\begin{theorem}\label{thm:soz_result_small_moduli}
If the GRH is true, $h=\sqrt{x}\log{x}$, and $x\geq x_0 \geq 1.05\cdot 10^7$ then there exist constants $\widetilde{k_1}(x_0)$ and $\widetilde{k_2}(x_0)$ such that
\begin{align*}
    \left|\sum_{\varrho_{\chi}} \frac{(x+h)^{\varrho_{\chi}+ 1} - x^{\varrho_{\chi}+ 1}}{h \varrho_{\chi}(\varrho_{\chi}+1)} \right| 
    \leq \left(\frac{\log{x}}{8\pi} + \frac{\log{q}}{2\pi} + \widetilde{k_1}(x_0)\right)\sqrt{x}\log{x} + \widetilde{k_2}(x_0)\sqrt{x}\log{q}.
\end{align*}
We present a selection of admissible $\widetilde{k_i}(x_0)$ for several values of $x_0$ in Table \ref{tab:soz_results}.
\end{theorem}

\begin{proof}
Assert $T = \sqrt{x}/\log{x}$ and re-trace the first steps of the proof of Theorem \ref{thm:soz_result} to see
\begin{align*}
    \left|\sum_{\varrho_{\chi}} \frac{(x+h)^{\varrho_{\chi}+ 1} - x^{\varrho_{\chi}+ 1}}{h \varrho_{\chi}(\varrho_{\chi}+1)} \right| 
    &\leq \sqrt{x} \left(1+\frac{h}{x}\right)^{\frac{1}{2}} \left(\omega
    + \sum_{200 \leq |\gamma_{\chi}| < \eta(x_0)} \frac{1}{\sqrt{\frac{1}{4} + \gamma_{\chi}^2}}
    + \sum_{\eta(x_0) \leq |\gamma_{\chi}| < T} \frac{1}{|\gamma_{\chi}|} \right) \\
    &\qquad\qquad\qquad\qquad\qquad\qquad\qquad+ \bigg(\left(1+\frac{h}{x}\right)^{\frac{3}{2}} + 1\bigg) \frac{x^{\frac{3}{2}}}{h} \sum_{|\gamma_{\chi}| \geq T} \frac{1}{\gamma_{\chi}^2} .
\end{align*}
Again, apply Lemma \ref{lem:Lehman_gen} with $\phi(t) = (1/4 + t^2)^{-1}$ to see
\begin{align}
    \sum_{200 \leq |\gamma_{\chi}| < \eta(x_0)} \frac{1}{\sqrt{\tfrac{1}{4} + \gamma_{\chi}^2}}  
    &\leq \left(\frac{0.494}{\sqrt{\tfrac{1}{4} + 200^2}} + \int_{200}^{\eta(x_0)} \frac{\pi^{-1}\,dt}{\sqrt{\tfrac{1}{4} + t^2}} \right) \log{q} 
    + \int_{200}^{\eta(x_0)} \frac{\pi^{-1} \log\frac{t}{2\pi}}{\sqrt{\frac{1}{4} + t^2}} \,dt \nonumber\\
    &\qquad\qquad + \frac{2\left(0.247\log{\frac{1}{400\pi}} + 6.894\right)}{\sqrt{\tfrac{1}{4} + 200^2}}  + 0.247 \int_{200}^{\eta(x_0)} \frac{t^{-1}\,dt}{\sqrt{\tfrac{1}{4} + t^2}} \nonumber\\
    &:= \widetilde{\nu_1}(x_0) \log{q} + \widetilde{\nu_2}(x_0) \label{eqn:soz_useful0_small_moduli}
\end{align}
Therefore, $h=\sqrt{x}\log{x}$, \eqref{eqn:soz_useful1}, \eqref{eqn:soz_useful2}, and \eqref{eqn:soz_useful0_small_moduli} imply
\begin{equation*}
\begin{split}
    \left|\sum_{\varrho_{\chi}} \frac{(x+h)^{\varrho_{\chi}+ 1} - x^{\varrho_{\chi}+ 1}}{h \varrho_{\chi}(\varrho_{\chi}+1)} \right| 
    \leq \left(\frac{\log{x}}{8\pi} + \frac{\log{q}}{2\pi}\right) \sqrt{x}\log{x}
    + f_3(x) \sqrt{x}\log{x} 
    &+ \widetilde{f_4}(x) \sqrt{x}\log{q} \\
    &\qquad\quad + \widetilde{f_5}(x) \sqrt{x} ,
\end{split}
\end{equation*}
where $f_3(x)$, $\nu_3(x)$, $\nu_4(x)$ were defined earlier,
\begin{align*}
    \widetilde{f_4}(x) &= \bigg(\left(1+\frac{\log{x}}{\sqrt{x}}\right)^{\frac{3}{2}} + 1\bigg) \left(\frac{1}{\pi} + \frac{0.494 \log{x}}{\sqrt{x}}\right) + \widetilde{\nu_1}(x) + \nu_3(x) , \quad\text{and}\\
    \widetilde{f_5}(x) &= (\omega + \widetilde{\nu_2}(x) + \nu_4(x)) \left(1+\frac{\log{x}}{\sqrt{x}}\right)^{\frac{1}{2}} - 0.5334 .
\end{align*}
This upper bound with $\omega= 21.664472\dots$ we have obtained here implies the result with 
\begin{equation*}
    \widetilde{k_1}(x_0) = f_3(x_0) + \frac{\widetilde{f_5}(x)}{\log{x_0}}
    \quad\text{and}\quad 
    \widetilde{k_2}(x_0) = \widetilde{f_4}(x) . \qedhere
\end{equation*}
\end{proof}

\section{The prime number theorem in short intervals}\label{sec:PNT4SI}

Recall that the prime number theorem tells us
\begin{equation*}
    \psi(x) \sim x, 
    \qquad\text{where}\qquad
    \psi(x) = \sum_{n\leq x} \Lambda(n),
\end{equation*}
so we should also expect $\psi(x+h) - \psi(x) \sim h$ for $h>0$; a relationship of this form is called the prime number theorem for short intervals. Cully-Hugill and Dudek established the latest explicit version of the prime number theorem for short intervals in \cite[Thm.~1]{CullyHugillDudek}; their corrected result is re-stated below.\footnote{There was a typo in \cite{CullyHugillDudek} that affected some of their constants; Cully-Hugill has corrected the typo and updated their result in her upcoming PhD thesis. Theorem \ref{thm:CH_D} is the result that will be stated in her thesis.}

\begin{theorem}[Cully-Hugill and Dudek]\label{thm:CH_D}
If $\sqrt{x}\log x\leq h \leq x^\frac{3}{4}$, $\log{x} \geq 40$, and the Riemann hypothesis is true, then
\begin{equation*}
    |\psi(x+h) - \psi(x)-h|< \frac{\sqrt{x} \log{x}}{\pi}  \log\bigg(\frac{h}{\sqrt{x}\log x} \bigg) + \frac{13}{6}\sqrt{x}\log{x} .
\end{equation*} 
\end{theorem}

Theorem \ref{thm:CH_D} would be an important ingredient in our proof of Theorem \ref{thm:exact_formula} later. However, we want to use the result in a broader range of $x$ and we will also only need to apply the result with $h = \sqrt{x}\log{x}$. Therefore, we prove the following refinement of Theorem \ref{thm:CH_D} in this special case and we will use that instead. All of the refinement we will obtain over Theorem \ref{thm:CH_D} stems from how we bound a particular sum over zeros.

\begin{theorem}\label{thm:CH_D_refined}
If the Riemann hypothesis is true and $x\geq x_0 \geq e^{10}$, then there exist constants $k_3(x_0), k_4(x_0) > 0$ such that
\begin{equation*}
    |\psi(x+\sqrt{x}\log{x}) - \psi(x) - \sqrt{x}\log{x}| < k_3(x_0)\sqrt{x}\log{x} - k_4(x_0) .
\end{equation*}
We present a selection of admissible $k_3(x_0)$, $k_4(x_0)$ for several values of $x_0\geq e^{10}$ in Table \ref{tab:CH_D_refined}.
\end{theorem}

\begin{table}[]
    \centering
    \begin{tabular}{c|ccccc}
        $\log{x_0}$ & $\kappa_0$ & $\kappa_1$ & $\kappa_2$ & $k_3(x_0)$ & $k_4(x_0)$ \\
        \hline
        $10$ & $0.05989$ & $18.81137$ & $1.74663$ & $1.86054$ & $8.31357\cdot 10^{3}$ \\
        $20$ & $0.0457$ & $37.77813$ & $1.74663$ & $1.19669$ & $3.15402\cdot 10^{6}$ \\
        $30$ & $0.03579$ & $52.1484$ & $1.86645$ & $0.98255$ & $8.10901\cdot 10^{8}$ \\
        $40$ & $0.03167$ & $63.91776$ & $1.74663$ & $0.85812$ & $1.64282\cdot 10^{11}$ \\
        $50$ & $0.02886$ & $74.8239$ & $2.15968$ & $0.7904$ & $3.10314\cdot 10^{13}$ \\
        $60$ & $0.02683$ & $85.00441$ & $2.28091$ & $0.73439$ & $5.59133\cdot 10^{15}$ \\
        $70$ & $0.02519$ & $94.2064$ & $2.37349$ & $0.69103$ & $9.74775\cdot 10^{17}$ \\
        $80$ & $0.02405$ & $102.33995$ & $2.46139$ & $0.65608$ & $1.63943\cdot 10^{20}$ \\
        $90$ & $0.02297$ & $111.44257$ & $2.56007$ & $0.62711$ & $2.7664\cdot 10^{22}$ \\
        $100$ & $0.02212$ & $120.2197$ & $2.65929$ & $0.60254$ & $4.58854\cdot 10^{24}$ \\
        $150$ & $0.01895$ & $157.01747$ & $2.97554$ & $0.51832$ & $5.00022\cdot 10^{35}$ \\
        $200$ & $0.01717$ & $189.4314$ & $3.2531$ & $0.46706$ & $4.77543\cdot 10^{46}$ \\
        $250$ & $0.01566$ & $222.13937$ & $3.47886$ & $0.43129$ & $4.4118\cdot 10^{57}$ \\
        $500$ & $0.01254$ & $347.59407$ & $4.35967$ & $0.33843$ & $1.66057\cdot 10^{112}$
    \end{tabular}
    \caption{Admissible values for $k_3(x_0)$, $k_4(x_0)$ at several choices of $x_0$, including parameter choices for $\kappa_i$ which inform these computations, in Theorem \ref{thm:CH_D_refined}.}
    \label{tab:CH_D_refined}
\end{table}

\begin{proof}
Suppose that $h = \sqrt{x}\log{x}$ and $2 \leq \tau < h$ is a parameter to be chosen. We use the same set-up as in \cite{CullyHugillDudek}, although our proof will refine their treatment of a sum over zeros that appears toward the end. Let
\begin{align*}
    w(n) = 
    \begin{cases}
        \frac{n - x + \tau}{\tau} & \text{if } x-\tau \leq n < x,\\
        1 &\text{if } x \leq n \leq x+h, \\
        \frac{x+h+\tau -n}{\tau} &\text{if } x+h < n \leq x+h+\tau ,\\
       0  &\text{otherwise.}
    \end{cases}
\end{align*}
It follows from $w(n) \leq 1$ that
\begin{align*}
    \left| \psi(x+h) - \psi(x) - \sum_{n} \Lambda(n) w(n) \right|
    &\leq \sum_{\substack{x-\tau \leq n < x\\\text{or}\\x+h < n \leq x+h+\tau}} \Lambda(n) 
\end{align*}
and
\begin{equation}\label{eqn:smoother_wn}
    \sum_{n} \Lambda(n) w(n) = \frac{\psi_1(x+h+\tau) - \psi_1(x+h) - \psi_1(x) + \psi_1(x-\tau)}{\tau} ,
\end{equation}
in which $\rho = 1/2 + i\gamma$ are the non-trivial zeros of $\zeta(s)$, $1.545 < \epsilon(x) < 2.069$, and
\begin{equation}\label{eqn:explicit_fmla_psi1}
    \psi_1(x) 
    = \int_2^x \psi(t)\,dt
    = \sum_{n\leq x} \Lambda(n) (x-n)
    = \frac{x^2}{2} - \sum_{\rho} \frac{x^{\rho+1}}{\rho (\rho +1)} - x \log(2\pi) + \epsilon(x) .
\end{equation}
The last explicit formula is a refinement of \cite[Lem.~4]{Dudek} that is described in \cite{CullyHugillDudek}. Insert \eqref{eqn:explicit_fmla_psi1} into \eqref{eqn:smoother_wn} to see
\begin{equation*}
    \left|\sum_{n} \Lambda(n) w(n) 
    - h - \tau 
    + \sum_{\rho} \frac{(x+h+\tau)^{\rho +1} - (x+h)^{\rho +1} - x^{\rho +1} + (x-\tau)^{\rho +1}}{\tau \rho(\rho +1)} \right|
    \leq \frac{21}{20\tau} .
\end{equation*}
Combine these observations to see
\begin{equation*}
    \left| \psi(x+h) - \psi(x) - h + \sum_{\rho} S_{\rho} \right|
    \leq \tau + \frac{21}{20\tau} + 2 \sum_{x+h < n \leq x+h+\tau} \Lambda(n),
\end{equation*}
in which
\begin{equation*}
    S_{\rho} = \frac{(x+h+\tau)^{\rho +1} - (x+h)^{\rho +1} - x^{\rho +1} + (x-\tau)^{\rho +1}}{\tau \rho(\rho +1)} .
\end{equation*}
Now, Cully-Hugill and Dudek used results from \cite{Broadbent, MontgomeryVaughan, Costa} to tell us that if $x \geq x_0 \geq e^{10}$, then 
\begin{align*}
    &\psi(x+h+\tau) - \psi(x+h) \\
    &\quad\leq \frac{2\tau\log(x+h+\tau)}{\log{\tau}} + \alpha_1 (x+h+\tau)^\frac{1}{2} + \alpha_2 (x+h+\tau)^\frac{1}{3} - 0.999(x+h)^\frac{1}{2} - \frac{2}{3} (x+h)^\frac{1}{3} \\
    &\quad\leq \frac{2\tau\log(x+h+\tau)}{\log{\tau}} + \beta_1 x^{\frac{1}{2}} + \beta_2 x^{\frac{1}{3}} ,
\end{align*}
in which $\alpha_1= 1+ 1.93378 \cdot 10^{-8}$, $\alpha_2 = 2.69$, $\beta_1 = \sqrt{3}\alpha_1 - 0.999$, and $\beta_2 = 3^{\frac{1}{3}}\alpha_2 - \frac{2}{3}$. It follows that for all $\log{x} \geq \log{x_0} \geq 10$, we have
\begin{equation}\label{eqn:at_the_end}
    \left| \psi(x+h) - \psi(x) - h + \sum_{\rho} S_{\rho} \right|
    \leq \tau + \frac{21}{20\tau} + \frac{4\tau\log(x+h+\tau)}{\log{\tau}} + 2 \beta_1 x^{\frac{1}{2}} + 2 \beta_2 x^{\frac{1}{3}} .
\end{equation}
All that remains is to bound the sum over zeros; we do this in a different way than \cite{CullyHugillDudek}. To achieve this goal, recall that $|\gamma| \geq \gamma_1 := 14.13472\dots$, let $\tau = \kappa_0\sqrt{x}\log{x}$ with $0<\kappa_0<1$, $\eta(x) = \sqrt{x} / \log{x}$, $\kappa_1 > 0$, and $\kappa_2 > 0$ such that $\kappa_2 \geq \kappa_1 \kappa_0$. Split the sum into four parts:
\begin{align*}
    \Sigma_1 = \sum_{\gamma_1 \leq |\gamma| \leq \kappa_1\eta(x_0)} S_{\rho}, \quad
    \Sigma_2(x) &= \sum_{\kappa_1\eta(x_0) < |\gamma| < \frac{\kappa_1 x}{h}} S_{\rho}, 
    \quad \\
    \Sigma_3(x) &= \sum_{\frac{\kappa_1 x}{h} \leq |\gamma| \leq \frac{\kappa_2 x}{\tau}} S_{\rho}, \quad\text{and}\quad 
    \Sigma_4(x) = \sum_{|\gamma| \geq \frac{\kappa_2 x}{\tau}} S_{\rho}. 
\end{align*}

To bound $\Sigma_4(x)$, import \cite[Lem.~1(ii)]{Skewes}, which tells us that for all $T \geq \gamma_1$, we have
\begin{equation*}
    \sum_{\gamma \geq T} \frac{1}{\gamma^2} < \frac{\log{T}}{2\pi T} .
\end{equation*}
It follows from $\eta(x) = x/h$ that
\begin{align}
    |\Sigma_4(x)|
    &< \frac{h \sqrt{x}}{\tau} \left(\left(1+\frac{h+\tau}{x}\right)^{\frac{3}{2}} + \left(1+\frac{h}{x}\right)^{\frac{3}{2}} + 2\right) \sum_{|\gamma|\geq \frac{\kappa_2 x}{\tau}} \frac{\eta(x)}{\gamma^2}\nonumber\\
    &< \frac{h \sqrt{x}}{\tau\pi} \left(\left(1+\frac{h+\tau}{x}\right)^{\frac{3}{2}} + \left(1+\frac{h}{x}\right)^{\frac{3}{2}} + 2\right) \frac{\eta(x)\tau}{\kappa_2 x} \log{\eta(x)} \nonumber\\
    &= \frac{\sqrt{x}}{\kappa_2\pi} \left(\left(1+\frac{h+\tau}{x}\right)^{\frac{3}{2}} + \left(1+\frac{h}{x}\right)^{\frac{3}{2}} + 2\right) \log{\frac{\kappa_2 \eta(x)}{\kappa_0}} \nonumber\\
    &= \frac{\sqrt{x}\log{x}}{\kappa_2\pi} \left(\left(1+\frac{h+\tau}{x}\right)^{\frac{3}{2}} + \left(1+\frac{h}{x}\right)^{\frac{3}{2}} + 2\right) \left(\frac{1}{2} - \frac{\log\log{x}}{\log{x}} + \frac{\log{\frac{\kappa_2}{\kappa_0}}}{\log{x}}\right) \nonumber\\
    &< \frac{\sqrt{x}\log{x}}{\kappa_2\pi} \left(\left(1+\frac{1+\kappa_0}{\eta(x_0)}\right)^{\frac{3}{2}} + \left(1+\frac{\kappa_0}{\eta(x_0)}\right)^{\frac{3}{2}} + 2\right) \left(\frac{1}{2} + \frac{\log{\frac{\kappa_2}{\kappa_0}}}{\log{x_0}}\right) \nonumber\\
    &:= \ell_0(x_0) \sqrt{x}\log{x} . \label{eqn:Sigma4}
\end{align}
To bound $\Sigma_2(x)$, recall \eqref{eqn:me_again}, which informs us that
\begin{equation*}
    \frac{T}{2\pi} \log{\frac{T}{2\pi e}} - \frac{7}{8} - R(T)
    \leq N(T) 
    < \frac{T\log{T}}{2\pi} .
\end{equation*}
It follows that
\begin{align}
    |\Sigma_2(x)|
    &\leq 2 ( N(\kappa_1 \eta(x)) - N(\kappa_1 \eta(x_0)) ) \nonumber\\
    &\leq \frac{\kappa_1 \eta(x) \log{\kappa_1 \eta(x)}}{\pi} - \frac{\kappa_1 \eta(x_0)}{\kappa_0\pi} \log{\frac{\kappa_1 \eta(x_0)}{2\pi e \kappa_0}} + \frac{7}{4} + 2 R(\kappa_1 \eta(x_0)) \nonumber\\
    &:= \frac{\kappa_1 \eta(x) \log{\kappa_1 \eta(x)}}{\pi} - \ell_1(x_0) . \label{eqn:Sigma2}
\end{align}
To bound the remaining sums, we will use the following bound: 
\begin{align*}
    |S_{\rho}| 
    = \left|\frac{\tau^{-1}}{\rho} \left(\int_{x+h}^{x+h+\tau} t^{\rho}\,dt - \int_{x-\tau}^{x} t^{\rho}\,dt\right)\right|
    &\leq \frac{\sqrt{x}}{|\rho|} \left(1 + \sqrt{1+\frac{1+\kappa_0}{\eta(x)}}\right)
    := \frac{\ell_2(x) \sqrt{x}}{|\rho|} .
\end{align*}
Next, Lemma \ref{lem:BPT} with $\phi(t) = t^{-1}$ implies
\begin{align}
    &|\Sigma_3(x)| 
    \leq \sum_{\frac{\kappa_1 x}{h} \leq |\gamma| \leq \frac{\kappa_2 x}{\tau}} \frac{\ell_2(x) \sqrt{x}}{|\gamma|} \nonumber\\
    &\quad < \frac{2\ell_2(x_0) \bigg(\frac{\left(\log{x} + \max\left\{0, \log{\frac{\kappa_1\kappa_2}{4\pi^2 \kappa_0}}\right\}\right)}{4\pi} \log{\frac{\kappa_2}{\kappa_0\kappa_1}} + \frac{\kappa_0 R(\frac{\kappa_2 \eta(x_0)}{\kappa_0})}{\kappa_2 \eta(x_0)} + \frac{R(\kappa_1 \eta(x_0))}{\kappa_1 \eta(x)} + \frac{4.200 + 4.134 \log{\kappa_1 \eta(x_0)}}{\kappa_1^2 \eta(x_0)^2} \bigg)}{(\sqrt{x}\log{x})^{-1} \log{x_0}} \nonumber\\
    &\quad := \ell_3(x_0)\sqrt{x}\log{x} . \label{eqn:Sigma3}
\end{align}
Finally, Lemma \ref{lem:BPT} with $\phi(t) = (\tfrac{1}{4} + t^2)^{-\tfrac{1}{2}}$ implies
\begin{align}
    &|\Sigma_1(x)| 
    \leq \sum_{\gamma_1 \leq |\gamma| \leq \kappa_1 \eta(x_0)} \frac{\ell_2(x) \sqrt{x}}{|\rho|} \nonumber\\
    &\,\,\,< \ell_2(x_0) \left(\frac{1}{\pi} \int_{\gamma_1}^{\kappa_1 \eta(x_0)} \frac{\log{\frac{t}{2\pi}}\,dt}{\sqrt{\tfrac{1}{4} + t^2}} + \frac{2 R(\kappa_1 \eta(x_0))}{\sqrt{\tfrac{1}{4} + \kappa_1^2 \eta(x_0)^2}} + \frac{2 R(\gamma_1)}{\sqrt{\tfrac{1}{4} + \gamma_1^2}} + 0.04509 \right) 
    := \ell_4(x_0) . \label{eqn:Sigma1}
\end{align}

Collect the observations \eqref{eqn:Sigma4}, \eqref{eqn:Sigma2}, \eqref{eqn:Sigma3}, and \eqref{eqn:Sigma1} together to see that if $x\geq x_0\geq e^{10}$, then
\begin{align*}
    \sum_{\rho} |S_{\rho}|
    &\leq \left(\ell_0(x_0) + \ell_3(x_0) + \frac{\kappa_1 \eta(x_0) \log{\kappa_1 \eta(x_0)}}{\pi\sqrt{x_0}\log{x_0}}\right) \sqrt{x}\log{x} - \ell_1(x_0) + \ell_4(x_0) \\
    &\leq \ell_5(x_0) \sqrt{x}\log{x} + \ell_6(x_0) .
\end{align*}
Moreover, the upper bound in \eqref{eqn:at_the_end} is majorised by $\ell_7(x_0) \sqrt{x}\log{x}$, where
\begin{align*}
    \ell_7(x) 
    &= \kappa_0 + \frac{21}{20\kappa_0 x(\log{x})^2} + \max\left\{8\kappa_0, \frac{4\kappa_0\log(x+(1+\kappa_0)\sqrt{x}\log{x})}{\log\left(\kappa_0\sqrt{x}\log{x}\right)}\right\} + \frac{2 \beta_1}{\log{x}} + \frac{2 \beta_2 x^{-\frac{1}{6}}}{\log{x}} .
\end{align*}
It follows from \eqref{eqn:at_the_end} and the preceding observations that 
\begin{equation*}
    k_3(x_0) = \ell_5(x_0) + \ell_7(x_0)
    \qquad\text{and}\qquad
    k_4(x_0) = - \ell_6(x_0)
\end{equation*}
are the constants in the statement of the theorem, so all that remains is to choose good $\kappa_0$, $\kappa_1$, and $\kappa_2$ for each $x_0$. Here, a good choice should ensure the value of $k_3(x_0)$ is minimised, since this term contributes the most to the upper bound. To reduce the complexity of this optimisation problem, we set $\kappa_2 = \max\{1.74663,\kappa_0\kappa_1\}$; this choice follows from the requirement $\kappa_0\kappa_1\geq\kappa_2$ and some computational experiments which displayed that \textit{small} $\kappa_2$ are desirable. Next, we use the \texttt{gp\_minimize} command from the \texttt{Python} package \texttt{skopt} to choose good $\kappa_0$ and $\kappa_1$ for each $x_0$. The choices that we make in the end, as well as the resulting values for $k_3(x_0)$ and $k_4(x_0)$, are presented in Table \ref{tab:CH_D_refined}.
\end{proof}

\begin{remark}
As it stands, the broadest range of $x$ we could prove Theorem \ref{thm:CH_D_refined} for is $\log{x} \geq 10$, since this is the range that the upper bound in \eqref{eqn:at_the_end} holds for. Of course, one could find values for $\alpha_i$ and $\beta_i$ such that the upper bound in \eqref{eqn:at_the_end} holds for a broader range of $x$, but $e^{10}$ is already quite small, and we are only concerned with larger values of $x$.
\end{remark}

\begin{remark}
The \texttt{gp\_minimize} command that we used to find optimised choices for $\kappa_i$ uses a Gaussian process to search its way toward good parameter choices (given some constraints) in a noisy function. The benefit of this optimiser is that it is relatively fast, straightforward to implement, and the outcome from a more thorough optimiser shouldn't be noticeably better than our outcomes in the end.
\end{remark}

\section{Bounds for twisted Chebyshev functions}\label{sec:twisted_psi}

Throughout this section, suppose that the GRH is true, $\varrho_{\chi} = 1/2 + i\gamma_{\chi}$ are the non-trivial zeros of $L(s,\chi)$, and $\chi\neq\chi_0$ modulo $q\geq 3$ is primitive. The primary purpose of this section is to prove Theorem \ref{thm:exact_formula}, which is the key ingredient in our proof of Corollary \ref{cor:pntpaps}. To begin, we observe that
\begin{equation}\label{eqn:Step1}
    \left|\psi(x,\chi) - \frac{\psi_1(x+h,\chi) - \psi_1(x,\chi)}{h} \right|  
    \leq \sum_{x < n \leq x+h} \Lambda(n) ,
\end{equation}
in which $h>0$ is a parameter to be chosen. Note that \eqref{eqn:Step1} follows from the relationship
\begin{equation*}
    \frac{\psi_1(x+h,\chi) - \psi_1(x,\chi)}{h}
    = \psi(x,\chi) + \sum_{x < n \leq x+h} \chi(n) \Lambda(n) \left(1 + \frac{x-n}{h}\right) .
\end{equation*}
It follows from \eqref{eqn:Step1} and Theorem \ref{thm:CH_D_refined} that for all $x\geq x_0\geq e^{10}$, we have
\begin{equation}\label{eqn:Step1wChoice}
    \left|\psi(x,\chi) - \frac{\psi_1(x+\sqrt{x}\log{x},\chi) - \psi_1(x,\chi)}{\sqrt{x}\log{x}} \right| < k_3(x_0)\sqrt{x}\log{x} - k_4(x_0) .
\end{equation} 
Next, if $x\geq 2$ and $\chi$ is principal or primitive, then \cite[Lem.~2.1]{DudekGrenieMolteni} tells us
\begin{equation*}
    \psi_1(x,\chi) 
    = \delta_{\chi} \frac{x^2}{2} - \sum_{\varrho_{\chi}} \frac{x^{\varrho_{\chi} +1}}{\varrho_{\chi}(\varrho_{\chi}+1)} 
    + \mathfrak{a}(\chi) x + \mathfrak{b}(\chi) + \mathfrak{c}(x,\chi) ,
\end{equation*}
in which $a_{\chi} = \frac{1-\chi(-1)}{2}$, $-\tfrac{L'(s,\chi)}{L'(s,\chi)} = - \tfrac{1}{s} + b_{\chi} + O(s)$, $-\tfrac{L'(s-1,\chi)}{L'(s-1,\chi)} = - \tfrac{1}{s} + c_{\chi} + O(s)$,
\begin{align*}
    \delta_{\chi} &= 
    \begin{cases}
        1 &\text{if } \chi = \chi_0, \\
        0 &\text{if } \chi \neq \chi_0, 
    \end{cases} \\
    \mathfrak{a}(\chi) 
    &= (1+b_{\chi}) (1-a_{\chi}) (1-\delta_{\chi}) - \frac{L'(0,\chi)}{L(0,\chi)} (a_{\chi}(1-\delta_{\chi}) + \delta_{\chi}), \\
    \mathfrak{b}(\chi) 
    &= (1-c_{\chi}) a_{\chi} + \frac{L'(-1,\chi)}{L(-1,\chi)} (1-a_{\chi}) , \quad\text{and}\\
    \mathfrak{c}(x,\chi) &= a_{\chi} \log{x} - (1-a_{\chi}) (1-\delta_{\chi}) x\log{x} - \sum_{m=1}^\infty\frac{x^{1-2m-a_{\chi}}}{(2m+a_{\chi})(2m-1+a_{\chi})}.
\end{align*}
It follows from this explicit formula for $\psi_1(x)$ that if $0<h<x$, then
\begin{align}
    &\frac{\psi_1(x+h,\chi) - \psi_1(x,\chi)}{h} \nonumber\\
    &\qquad= \frac{1}{h} \left(\delta_{\chi} \frac{(x+h)^2 - x^2}{2} - \sum_{\varrho_{\chi}} \frac{(x+h)^{\varrho_{\chi} +1} - x^{\varrho_{\chi} +1}}{\varrho_{\chi}(\varrho_{\chi}+1)} 
    + \mathfrak{a}(\chi) h + \mathfrak{c}(x+h,\chi) - \mathfrak{c}(x,\chi)\right) \nonumber\\
    &\qquad= \delta_{\chi}\left(x + \frac{h}{2}\right) + \mathfrak{a}(\chi) + \frac{1}{h} \left(\mathfrak{c}(x+h,\chi) - \mathfrak{c}(x,\chi) - \sum_{\varrho_{\chi}} \frac{(x+h)^{\varrho_{\chi} +1} - x^{\varrho_{\chi} +1}}{\varrho_{\chi}(\varrho_{\chi}+1)} \right) . \label{eqn:expression} 
\end{align}
Using \eqref{eqn:expression}, we will prove the following result in Section \ref{ssec:smoothed_difference}; this will be an important ingredient in our proof of of Theorem \ref{thm:psi_chi_expl_fmla} (below), which unlocks our proof of Theorem \ref{thm:exact_formula}.

\begin{theorem}\label{thm:smoothed_difference}
If $\chi$ modulo $q\geq 3$ is primitive and non-principal, $x\geq x_0 \geq \max\{e^{10},q\}$, and $h=\sqrt{x}\log{x}$, then there exist constants $k_5(x_0)$, $k_6(x_0)$ (which will be defined in the proof) such that
\begin{equation*}
    \left|\frac{\psi_1(x+h,\chi) - \psi_1(x,\chi)}{h}\right| 
    < \left(\frac{\log{x}}{8\pi} + \frac{\log{q}}{2\pi} + k_5(x_0)\right) \sqrt{x}\log{x} + k_6(x_0)\sqrt{x} + 1.777 .
\end{equation*}
Admissible values for $k_5(x_0)$ at several values of $x_0\geq e^{10}$ are presented in Table \ref{tab:smoothed_difference}.
\end{theorem}

\begin{table}[]
    \centering
    \begin{tabular}{c|ccccc}
        $\log{x_0}$ & $k_5(x_0)$ & $k_6(x_0)$ & $\Omega_0(x_0)$ & $\Omega_1(x_0)$ & $\Omega_2(x_0)$ \\
        \hline
        $10$ & $5.872$ & $0.0$ & $7.73253$ & $7.8834\cdot 10^{-1}$ & $-8.31179\cdot 10^{3}$ \\
        $20$ & $3.69604$ & $0.0$ & $4.89272$ & $2.07932\cdot 10^{-2}$ & $-3.15402\cdot 10^{6}$ \\
        $30$ & $3.24308$ & $0.0$ & $4.22562$ & $3.12938\cdot 10^{-4}$ & $-8.10901\cdot 10^{8}$ \\
        $40$ & $3.0183$ & $0.0$ & $3.87641$ & $3.73481\cdot 10^{-6}$ & $-1.64282\cdot 10^{11}$ \\
        $50$ & $2.88272$ & $0.0$ & $3.67311$ & $3.92334\cdot 10^{-8}$ & $-3.10314\cdot 10^{13}$ \\
        $60$ & $2.79193$ & $0.0$ & $3.52631$ & $3.80107\cdot 10^{-10}$ & $-5.59133\cdot 10^{15}$ \\
        $70$ & $2.72685$ & $0.0$ & $3.41787$ & $3.48232\cdot 10^{-12}$ & $-9.74775\cdot 10^{17}$ \\
        $80$ & $2.67781$ & $0.0$ & $3.33388$ & $3.06221\cdot 10^{-14}$ & $-1.63943\cdot 10^{20}$ \\
        $90$ & $2.63006$ & $0.0$ & $3.25716$ & $2.60976\cdot 10^{-16}$ & $-2.7664\cdot 10^{22}$ \\
        $100$ & $2.06796$ & $0.0$ & $2.67049$ & $2.16984\cdot 10^{-18}$ & $-4.58854\cdot 10^{24}$ \\
        $150$ & $-0.39621$ & $-6.69263$ & $0.1221$ & $-6.69263\cdot 10^{0}$ & $-5.00022\cdot 10^{35}$ \\
        $200$ & $-1.72212$ & $-15.33454$ & $-1.25507$ & $-1.53345\cdot 10^{1}$ & $-4.77543\cdot 10^{46}$ \\
        $250$ & $-3.3014$ & $-23.99894$ & $-2.87012$ & $-2.39989\cdot 10^{1}$ & $-4.4118\cdot 10^{57}$ \\
        $500$ & $-12.37591$ & $-67.46897$ & $-12.03749$ & $-6.7469\cdot 10^{1}$ & $-1.66057\cdot 10^{112}$
    \end{tabular}
    \caption{Admissible values for $k_5(x_0)$, $k_6(x_0)$, $\Omega_0(x_0)$, $\Omega_1(x_0)$, and $\Omega_2(x_0)$ at several choices of $x_0$ for Theorems \ref{thm:smoothed_difference} and \ref{thm:psi_chi_expl_fmla}.}
    \label{tab:smoothed_difference}
\end{table}

Next, Ernvall-Hyt\"{o}nen and Paloj\"{a}rvi proved in \cite[Lem.~13]{ErnvallHytonenPalojarvi} that if $\chi$ is a non-primitive Dirichlet character modulo $q\geq 3$ that is induced by a primitive character $\chi^*$, then
\begin{equation}\label{eqn:psiNonPrimitive}
    \frac{\left|\psi_0(x,\chi)-\psi_0(x,\chi^*)\right|}{\log{x}}
    \leq 1.12\log{q} ,
    \qquad\text{where}\qquad
    \psi_0(x,\chi) = \frac{\psi(x^+,\chi)+\psi(x^-,\chi)}{2} ;
\end{equation}
this modification to the definition of $\psi(x,\chi)$ accounts for discontinuities that may be present at prime powers. A simple consequence of this definition is that
\begin{equation}\label{eqn:psi_psi_0}
    |\psi(x,\chi) - \psi_0(x,\chi)| \leq \frac{\log{x}}{2} .
\end{equation}
Finally, insert Theorem \ref{thm:smoothed_difference} into \eqref{eqn:Step1wChoice} and apply \eqref{eqn:psiNonPrimitive}, \eqref{eqn:psi_psi_0} to extend the result to any non-principal $\chi$; the end result follows.

\begin{theorem}\label{thm:psi_chi_expl_fmla}
If $x\geq x_0\geq \max\{e^{10},q\}$ and $\chi\neq\chi_0$ modulo $q\geq 3$, then
\begin{equation*}
\begin{split}
    |\psi(x,\chi)| 
    &< \left(\frac{\log{x}}{8\pi} + \frac{\log{q}}{2\pi} + \Omega_0(x_0)\right) \sqrt{x}\log{x} + \Omega_1(x_0) \sqrt{x} + \Omega_2(x_0) ,
\end{split}
\end{equation*}
in which $\Omega_0(x_0) = k_3(x_0) + k_5(x_0)$, $\Omega_2(x_0) = 1.777 - k_4(x_0)$, and
\begin{equation*}
    \Omega_1(x) = k_6(x) + \frac{(0.5 + 1.12\log{x})\log{x}}{\sqrt{x}} .
\end{equation*}
We present a selection of admissible $\Omega_i(x_0)$ for several values of $x_0\geq e^{10}$ in Table \ref{tab:smoothed_difference}. 
\end{theorem}

Using Theorem \ref{thm:psi_chi_expl_fmla}, we will prove Theorem \ref{thm:exact_formula} (below) in Section \ref{ssec:objective}. As a bonus, we also prove the following corollary of Theorem \ref{thm:exact_formula} in Section \ref{ssec:corollary1}; this is a simple extension that makes an important ingredient in \cite{LanguascoZaccagnini} explicit.

\begin{theorem}\label{thm:exact_formula}
If the GRH is true, $q\geq 3$, and $x\geq x_0 \geq\max\{e^{10},q\}$, then 
\begin{equation*}
    |\psi(x,\chi) - \delta_{\chi}x| \leq 
    \begin{cases}
        \frac{\sqrt{x} (\log{x})^2}{8\pi} + 1.12 (\log{x})^2 &\text{if $\delta_{\chi}=1$,} \\
        \left(\frac{\log{x}}{8\pi} + \frac{\log{q}}{2\pi} + \Omega_0(x_0)\right) \sqrt{x}\log{x} + \Omega_1(x_0) \sqrt{x} + \Omega_2(x_0) &\text{if $\delta_{\chi}=0$.}
    \end{cases}
\end{equation*}
\end{theorem}

\begin{corollary}\label{cor:exact_formula}
If the GRH is true, $x\geq \max\{e^{10},q\}$, and $q\geq 3$, then 
\begin{equation*}
    |\theta(x,\chi) - \delta_{\chi}x| \leq 
    \begin{cases}
        \frac{\sqrt{x} (\log{x})^2}{8\pi} + 1.44270\sqrt{x}\log{x} + 1.12 (\log{x})^2 &\text{if $\delta_{\chi}=1$,} \\
        \left(\frac{\log{x}}{8\pi} + \frac{\log{q}}{2\pi} + 1.44270 + \Omega_0(x_0)\right) \sqrt{x}\log{x} + \Omega_1(x_0) \sqrt{x} + \Omega_2(x_0) &\text{if $\delta_{\chi}=0$.}
    \end{cases}
\end{equation*}
\end{corollary}

\begin{remark}
Using \eqref{eqn:smoothed_difference_precise}, we can remove the requirement $x\geq q$ from Theorems \ref{thm:smoothed_difference}-\ref{thm:psi_chi_expl_fmla}, as long as $q$ is in an explicit range. In fact, if $q< 10^{30}$ or $q\leq 10\,000$, then \eqref{eqn:what_limits1} and \eqref{eqn:what_limits2} respectively provide explicit upper bounds that hold for $x\geq x_0\geq e^{10}$ and $x\geq x_0\geq 1.05\cdot 10^7$. Using \eqref{eqn:what_limits2}, we will give a refinement of Theorem \ref{thm:psi_chi_expl_fmla} in Theorem \ref{thm:psi_chi_expl_fmla_small_moduli} that holds for $x\geq x_0\geq 1.05\cdot 10^7$ and $q\leq 10\,000$. However, to obtain general upper bounds, whose secondary terms depend on $x$ only, we needed to assert $x\geq q$.
\end{remark}

\begin{remark}
At face value, it appears as though we could improve the outcome of Theorem \ref{thm:psi_chi_expl_fmla} (and each of it's consequences) by noting that \eqref{eqn:Step1} can be refined into
\begin{equation*}
    \left|\psi(x,\chi) - \frac{\psi_1(x+h,\chi) - \psi_1(x,\chi)}{h} \right|  
    \leq \sum_{\substack{x < n \leq x+h\\(n,q)=1}} \Lambda(n) ,
\end{equation*}
taking $h=\sqrt{x}\log{x}$ as usual. To see what we could achieve, note that
\begin{equation*}
    \sum_{\substack{n\leq x\\(n,q)=1}} \Lambda(n)
    = \sum_{n\leq x} \Lambda(n) \sum_{d|(n,q)} \mu(d)
    = \sum_{d|q} \mu(d) \sum_{n\leq \frac{x}{d}} \Lambda(n)
    \quad\text{and}\quad 
    \sum_{d|q} \frac{\mu(d)}{d} = \frac{\varphi(q)}{q};
\end{equation*}
here $\mu$ denotes the M\"{o}bius function. Use these observations and consider what we expect from the prime number theorem in short intervals to see
\begin{equation*}
    \sum_{\substack{x<n\leq x+h\\(n,q)=1}} \Lambda(n)
    = \sum_{d|q} \mu(d) \sum_{\frac{x}{d} < n \leq \frac{x+h}{d}} \Lambda(n) \sim h \sum_{d|q} \frac{\mu(d)}{d} = \frac{\varphi(q)}{q} h
    \quad\text{as}\quad
    x\to\infty .
\end{equation*}
The coefficient $\varphi(q)/q\to 1^-$ as $q\to\infty$ when prime $q$ are considered, whereas $\varphi(q)/q \to 0^+$ as $q\to\infty$ when highly composite $q$ (such as primorials) are observed. Therefore, our ``less refined'' bound in \eqref{eqn:Step1} is not as bad as one might expect, since we are considering general $q$.
\end{remark}

\subsection{Preliminary bounds}

We will need the following lemmas in what follows.

\begin{lemma}\label{lem:Lemmabchi}
If $\chi$ is a primitive, non-principal character modulo $q$ with $\chi(-1)=1$, then
\begin{equation*}
    |b_{\chi}| <
    \begin{cases}
        2.751\log{q} + 23.878 &\text{if }q\geq 3,\\
        \log{q} + 2\log\log{q} - 0.224 &\text{if }q\geq 10^{30}.
    \end{cases}
\end{equation*}
\end{lemma}

\begin{proof}
The first case is \cite[Lem.~17]{ErnvallHytonenPalojarvi} and the second is \cite[Cor.~2]{ChirreSimonicHagen}.
\end{proof}

\begin{lemma}\label{lem:logarithmicL}
If $\chi$ is primitive modulo $q\geq 3$ such that $\chi(-1)=-1$, then
\begin{equation*}
    \left|\frac{L'(0,\chi)}{L(0,\chi)}\right| 
    < \mathfrak{f}(q)
    < 316.5 + 0.593\log\log{q} (\log{q})^2 + 0.0758 \sqrt{q}\log{q} + \log{q},
\end{equation*}
where
\begin{equation*}
\mathfrak{f}(q) := \gamma + \log{2} + \left|\log{\frac{q}{\pi}}\right| +
\begin{cases}
    0.027\sqrt{q}\log{q}+0.067\sqrt{q}+316.229 &\text{if $3\leq q<4\cdot 10^5$,} \\
    3.715(\log{q})^2 &\text{if $4\cdot 10^5 \leq q < 10^{10}$,} \\
    (0.593\log\log q+1.205)(\log{q})^2 &\text{if $q\geq 10^{10}$.}
    \end{cases}
\end{equation*}
\end{lemma}

\begin{proof}
Apply \cite[Cor.~20]{ErnvallHytonenPalojarvi} to obtain the first inequality. For all $q\geq 10^{10}$, we have
\begin{equation}\label{eqn:fq_large_q}
    \mathfrak{f}(q) 
    < \gamma + \log{\frac{2}{\pi}} + 0.593\log\log{q} (\log{q})^2 + 0.000278 \sqrt{q}\log{q} + \log{q} .
\end{equation}
For all $4\cdot 10^{5}\leq q < 10^{10}$, we have
\begin{align*}
    \mathfrak{f}(q) 
    &< \gamma + \log{\frac{2}{\pi}} + 0.0758 \sqrt{q}\log{q} + \log{q} .
\end{align*}
For all $3 < q < 4\cdot 10^{5}$, we have
\begin{equation*}
    \mathfrak{f}(q) 
    < \gamma + \log{\frac{2}{\pi}} + 316.229 + 0.0754 \sqrt{q}\log{q} + \log{q} .
\end{equation*}
It follows that for all $q>3$, we have
\begin{equation*}
    \mathfrak{f}(q) < 316.5 + 0.593\log\log{q} (\log{q})^2 + 0.0758 \sqrt{q}\log{q} + \log{q}
\end{equation*}
Moreover, if $q=3$, then $\mathfrak{f}(q) \approx 317.72$, which is also majorised by the preceding bound.
\end{proof}

\begin{lemma}\label{lem:sumx}
If $x\geq 2$, then
\begin{align*}
    \sum_{m=1}^\infty\frac{x^{1-2m-a_{\chi}}}{(2m+a_{\chi})(2m-1+a_{\chi})} 
    = a_{\chi} + (-x)^{a_{\chi}} \tanh^{-1}(x^{-1}) - x^{1-a_{\chi}} \log{\sqrt{1-x^{-2}}} . 
\end{align*}
\end{lemma}

\begin{proof}
If $\chi(-1)=-1$, then $a_{\chi}=1$ and
\begin{align*}
    \sum_{m=1}^\infty\frac{t^{2m}}{2m(2m+1)} \bigg|_{t=\tfrac{1}{x}} 
    = \int_0^{1/x} t^{-2} \sum_{m=1}^\infty \frac{t^{2m+1}}{2m+1}\,dt 
    &= \int_0^{1/x} \frac{\tanh^{-1}{t} - t}{t^2}\,dt \\
    &= 1 -x\tanh^{-1}(x^{-1}) - \log{\sqrt{1-x^{-2}}} .
\end{align*}
Similarly, if $\chi(-1)=1$, then $a_{\chi}=0$ and
\begin{align*}
    \sum_{m=1}^\infty\frac{t^{2m-1}}{2m(2m-1)} \bigg|_{t=\tfrac{1}{x}} 
    = \int_0^{1/x} t^{-2} \sum_{m=1}^\infty \frac{t^{2m}}{2m}\,dt 
    &= - \int_0^{1/x} \frac{\log{\sqrt{1-t^2}}}{t^2}\,dt \\
    &= \tanh^{-1}(x^{-1}) - x\log{\sqrt{1-x^{-2}}} . \qedhere
\end{align*}
\end{proof}

\begin{lemma}\label{lem:Schoenfeld}
If the Riemann hypothesis is true and $x\geq 73.2$, then
\begin{equation*}
    |\psi(x) - x| < \frac{\sqrt{x} (\log{x})^2}{8\pi} .
\end{equation*}
\end{lemma}

\begin{proof}
This is \cite[(6.2)]{Schoenfeld}.
\end{proof}

\begin{lemma}\label{lem:LambdaSyt}
If $x\geq 2$ and $q\geq 3$, then
\begin{equation*}
    \frac{1}{\log{x}}\sum_{\substack{n\leq x \\ (n,q)>1}}\Lambda(n)
    \leq \tau(q) 
    \leq 1.12\log{q} ,
    \quad\text{where}\quad
    \tau(q) =  
    \begin{cases}
        2 & \text{if }q=6,\\ 
        \log{q} & \text{if }q\neq 6.
    \end{cases}
\end{equation*}
\end{lemma}

\begin{proof}
This is \cite[Lem.~12]{ErnvallHytonenPalojarvi}.
\end{proof}

\subsection{Proof of Theorem \ref{thm:smoothed_difference}}\label{ssec:smoothed_difference}

Using Lemmas \ref{lem:Lemmabchi}-\ref{lem:sumx}, we establish the following ingredients.

\begin{lemma}\label{lem:a_chi}
If $\chi$ is non-principal and primitive modulo $q\geq 3$, then
\begin{equation*}
\begin{split}
    |\mathfrak{a}(\chi)| 
    &\leq 
    \begin{cases}
        2.751\log{q} + 24.878 &\text{if $a_{\chi}=0$}\\
        0.593\log\log{q} (\log{q})^2 + 0.0758 \sqrt{q}\log{q} + \log{q} + 316.5, &\text{if $a_{\chi}=1$}
    \end{cases} \\
    &< 0.593\log\log{q} (\log{q})^2 + 0.0758 \sqrt{q}\log{q} + 2.751\log{q} + 316.5 .
\end{split}
\end{equation*}
Moreover, if $q\geq 10^{30}$, then we can refine the preceding bound into
\begin{equation*}
    |\mathfrak{a}(\chi)| 
    < 0.593\log\log{q} (\log{q})^2 + 0.000278 \sqrt{q}\log{q} + \log{q} + 0.776 .
\end{equation*}
\end{lemma}

\begin{proof}
If $\chi$ is non-principal, then
\begin{equation*}
    \mathfrak{a}(\chi) =
    \begin{cases}
        1+b_{\chi} &\text{if $\chi(-1) = 1$,}\\
        - \frac{L'(0,\chi)}{L(0,\chi)} &\text{if $\chi(-1) = -1$.}
    \end{cases}
\end{equation*}
Apply Lemma \ref{lem:Lemmabchi} when $\chi(-1)=1$ and Lemma \ref{lem:logarithmicL} when $\chi(-1)=-1$. The refinement follows easily from \eqref{eqn:fq_large_q}.
\end{proof}

\begin{lemma}\label{lem:c_chi_diff}
If $\log{x}\geq 10$, $h=\sqrt{x}\log{x}$, and $\chi\neq\chi_0$ modulo $q\geq 3$ is primitive, then
\begin{equation*}
    \left|\frac{\mathfrak{c}(x+h,\chi) - \mathfrak{c}(x,\chi)}{h}\right|
    \leq \frac{a_{\chi}}{x} + (1-a_{\chi})\left(\log{x} + 1 + \frac{\log{x}}{\sqrt{x}}\right) + \frac{7\cdot 10^{-5}}{h} .
\end{equation*}
\end{lemma}

\begin{proof}
For any $0<h<x$, we have
\begin{align*}
    \frac{\mathfrak{c}(x+h,\chi) - \mathfrak{c}(x,\chi)}{h} 
    &= \frac{a_{\chi}}{h} \log\!\left(1+\frac{h}{x}\right) 
    - (1-a_{\chi}) \left(\frac{x+h}{h}\log\!\left(1+\frac{h}{x}\right) + \log{x}\right) \\
    &\qquad\qquad+ \sum_{m=1}^\infty \left(1 - \left(1+\frac{h}{x}\right)^{1-2m-a_{\chi}}\right) \frac{x^{1-2m-a_{\chi}}}{(2m+a_{\chi})(2m-1+a_{\chi}) h} .
\end{align*}
Apply Lemma \ref{lem:sumx}, $\log(1+x) < x$ for all $x>0$, and $h=\sqrt{x}\log{x}$ to see
\begin{equation*}
\begin{split}
    \left|\frac{\mathfrak{c}(x+h,\chi) - \mathfrak{c}(x,\chi)}{h}\right|
    &\leq \frac{a_{\chi}}{x} + (1-a_{\chi})\left(\log{x} + 1 + \frac{\log{x}}{\sqrt{x}}\right) \\
    &\qquad\qquad\qquad+ \frac{a_{\chi} + (-x)^{a_{\chi}} \tanh^{-1}(x^{-1}) - x^{1-a_{\chi}} \log{\sqrt{1-x^{-2}}}}{h} .
\end{split}
\end{equation*}
To complete the proof, observe that the numerator in the final fraction decreases in $x$ whether $a_{\chi} = 0$ or $a_{\chi} = 1$, so it is is majorised for all $\log{x}\geq 10$ by
\begin{equation*}
    \max\{7\cdot 10^{-5}, 4\cdot 10^{-10}\} = 7\cdot 10^{-5} . \qedhere
\end{equation*}
\end{proof}


Insert Theorem \ref{thm:soz_result}, Lemma \ref{lem:a_chi}, and Lemma \ref{lem:c_chi_diff} into \eqref{eqn:expression} to see that if $\chi$ is a non-principal and primitive Dirichlet character modulo $q\geq 3$, $h=\sqrt{x}\log{x}$, and $x\geq x_0 \geq e^{10}$, then
\begin{align}
    &\left|\frac{\psi_1(x+h,\chi) - \psi_1(x,\chi)}{h} \right| \leq |\mathfrak{a}(\chi)| + \left|\frac{\mathfrak{c}(x+h,\chi) - \mathfrak{c}(x,\chi)}{h}\right| + \bigg|\sum_{\varrho_{\chi}} \frac{(x+h)^{\varrho_{\chi} +1} - x^{\varrho_{\chi} +1}}{h \varrho_{\chi}(\varrho_{\chi}+1)} \bigg| \nonumber\\
    &\qquad\qquad\qquad\qquad< \left(\frac{\log{x}}{8\pi} + \frac{\log{q}}{2\pi} + g_1(x)\right) \sqrt{x}\log{x} + k_2(x_0) \sqrt{x}\log{q} + g_2(q) , \label{eqn:smoothed_difference_precise}
\end{align}
in which
\begin{align*}
    g_1(x) &= k_1(x_0) + \frac{1}{\sqrt{x}} + \frac{1}{x}, \\
    g_2(q) &= 
    \begin{cases}
        317.501 + 0.593\log\log{q} (\log{q})^2 + 0.0758 \sqrt{q}\log{q} + 2.751\log{q} &\text{if }q<10^{30}, \\
        1.777 + 0.593\log\log{q} (\log{q})^2 + 0.000278 \sqrt{q}\log{q} + \log{q} &\text{if }q\geq 10^{30}.
    \end{cases}
\end{align*}
Now that we have proved this important identity, we are able to complete the result using the three following observations.

\medskip\noindent\textsc{Observation I:} 
First, we restrict our attention to $3 \leq q_0 \leq q < 10^{30}$. Since $g_1(x)$ decreases in $x$, for all $x\geq x_0\geq e^{10}$, we also have
\begin{equation}\label{eqn:what_limits}
    \left|\frac{\psi_1(x+h,\chi) - \psi_1(x,\chi)}{h} \right|
    < \left(\frac{\log{x}}{8\pi} + \frac{\log{q}}{2\pi} + \varsigma_1(x)\right) \sqrt{x}\log{x} + \varsigma_2(x_0) \sqrt{x} ,
\end{equation}
in which 
\begin{equation*}
    \varsigma_1(x) = k_1(x_0) + \frac{1}{\sqrt{x}} + \frac{1}{x} + \frac{g_2(10^{30}-1)}{\sqrt{x}\log{x}}
    \qquad\text{and}\qquad
    \varsigma_2(x_0) =
    \begin{cases}
        30 k_2(x_0) \log{10} & \text{if }k_2(x_0) \geq 0, \\
        k_2(x_0) \log{q_0} & \text{if }k_2(x_0) < 0.
    \end{cases}
\end{equation*}

\medskip\noindent\textsc{Observation II:} 
Second, we restrict our attention to $q \geq 10^{30}$. If $x\geq x_0\geq \max\{e^{10},q\}$, then we may refine \eqref{eqn:smoothed_difference_precise} to tell us
\begin{equation}\label{eqn:what_limits1}
\begin{split}
    &\left|\frac{\psi_1(x+h,\chi) - \psi_1(x,\chi)}{h} \right| \\
    &\qquad\qquad< \left(\frac{\log{x}}{8\pi} + \frac{\log{q}}{2\pi} + \varsigma_3(x)\right) \sqrt{x}\log{x} + 30 \min\{k_2(x_0),0\} \sqrt{x} \log{10} + 1.777,
\end{split}
\end{equation}
in which 
\begin{equation*}
    \varsigma_3(x) = \frac{0.593 \log\log{x} \log{x}}{\sqrt{x}} + k_1(x_0) + \max\{k_2(x_0),0\} + 0.000278 + \frac{2}{\sqrt{x}} + \frac{1}{x}.
\end{equation*}

\medskip\noindent\textsc{Observation III:} 
Finally, suppose $q\geq 3$. If $x\geq x_0\geq \max\{e^{10},q\}$, then
\begin{equation}\label{eqn:what_limits1a}
    \left|\frac{\psi_1(x+h,\chi) - \psi_1(x,\chi)}{h} \right| < \left(\frac{\log{x}}{8\pi} + \frac{\log{q}}{2\pi} + \varsigma_4(x)\right) \sqrt{x}\log{x} + 1.777,
\end{equation}
in which 
\begin{equation*}
    \varsigma_4(x) = \frac{0.593 \log\log{x} \log{x}}{\sqrt{x}} + k_1(x_0) + \max\{k_2(x_0),0\} + 0.0758 + \frac{3.751}{\sqrt{x}} + \frac{1}{x} + \frac{315.724}{\sqrt{x}\log{x}}.
\end{equation*}
 
\medskip\noindent\textsc{Completion:} 
Experimental computations tell us that \eqref{eqn:what_limits1a} will yield a better numerical outcome when $k_2(x_0) \geq 0$ and combining \eqref{eqn:what_limits}, \eqref{eqn:what_limits1} will yield the better outcome when $k_2(x_0) < 0$. Therefore, it follows from \eqref{eqn:what_limits}, \eqref{eqn:what_limits1}, and \eqref{eqn:what_limits1a} that the result holds with
\begin{align*}
    k_5(x) &= 
    \begin{cases}
        \varsigma_4(x) & \text{if }k_2(x_0) \geq 0 , \\
        \varsigma_5(x) & \text{if }k_2(x_0) < 0 ,
    \end{cases}
    \quad\text{and}\quad
    k_6(x) = 
    \begin{cases}
        0 & \text{if }k_2(x_0) \geq 0 ,\\
        k_2(x_0) \log{3} & \text{if }k_2(x_0) < 0 .
    \end{cases}
\end{align*}
in which 
\begin{align*}
    \varsigma_5(x) 
    &= k_1(x_0) + 0.000278 + \frac{2}{\sqrt{x}} + \frac{1}{x} + \frac{\max\{g_2(10^{30}-1), 0.593\log\log{x} (\log{x})^2\}}{\sqrt{x}\log{x}} .
\end{align*}

\begin{remark}
If we restrict our attention to the smallest moduli $3\leq q \leq 10^4$, then we can use Theorem \ref{thm:soz_result_small_moduli} (instead of Theorem \ref{thm:soz_result}) in the preceding logic to see that 
\begin{equation}\label{eqn:what_limits2}
    \left|\frac{\psi_1(x+h,\chi) - \psi_1(x,\chi)}{h} \right|
    < \left(\frac{\log{x}}{8\pi} + \frac{\log{q}}{2\pi} + \varsigma_6(x)\right) \sqrt{x}\log{x} + \varsigma_7(x_0) \sqrt{x} 
\end{equation}
for all $x\geq x_0\geq 1.05\cdot 10^7$, in which 
\begin{equation*}
    \varsigma_6(x) = \widetilde{k_1}(x_0) + \frac{1}{\sqrt{x}} + \frac{1}{x} + \frac{g_2(10^4)}{\sqrt{x}\log{x}}
    \qquad\text{and}\qquad
    \varsigma_7(x_0) =
    \begin{cases}
        4 \widetilde{k_2}(x_0) \log{10} & \text{if }\widetilde{k_2}(x_0) \geq 0, \\
        \widetilde{k_2}(x_0) \log{3} & \text{if }\widetilde{k_2}(x_0) < 0.
    \end{cases}
\end{equation*}
\end{remark}

\subsection{A refinement}

It is clear from the preceding case analyses that we only require $x_0\geq q$ in some cases. Therefore, following the logic at the beginning of this section, one can use \eqref{eqn:what_limits2} to refine Theorem \ref{thm:psi_chi_expl_fmla} into the following result for a fixed range of $q$. 

\begin{theorem}\label{thm:psi_chi_expl_fmla_small_moduli}
If $x\geq x_0\geq 1.05\cdot 10^7$, the GRH is true, and $\chi\neq\chi_0$ modulo $q$ such that $3\leq q \leq 10\,000$, then
\begin{equation*}
\begin{split}
    |\psi(x,\chi)| 
    &< \left(\frac{\log{x}}{8\pi} + \frac{\log{q}}{2\pi} + \widetilde{\Omega_0}(x_0)\right) \sqrt{x}\log{x} + \widetilde{\Omega_1}(x_0) \sqrt{x} + \widetilde{\Omega_2}(x_0) ,
\end{split}
\end{equation*}
in which $\widetilde{\Omega_0}(x) = k_3(x) + \varsigma_6(x)$, $\widetilde{\Omega_2}(x) = - k_4(x)$, and
\begin{equation*}
    \widetilde{\Omega_1}(x) = \varsigma_7(x) + \frac{(0.5 + 1.12\log{x})\log{x}}{\sqrt{x}} .
\end{equation*}
\end{theorem}

\subsection{Proof of Theorem \ref{thm:exact_formula}}\label{ssec:objective}

If $x\geq 2$ and $\chi = \chi_0$ modulo $q\geq 3$, then
\begin{equation}\label{eqn:Step0}
    \left|\psi(x,\chi_0)-\psi(x)\right| \leq \sum\limits_{\substack{n \leq x \\ (n,q)>1}} \Lambda(n),
    \qquad\text{where}\qquad
    \psi(x)=\sum_{n\leq x}\Lambda(n) .
\end{equation}
Insert Lemmas \ref{lem:Schoenfeld}-\ref{lem:LambdaSyt} into \eqref{eqn:Step0} to see
\begin{equation}\label{eqn:principal_chi_case}
    \left|\psi(x,\chi_0)-x\right| \leq \frac{\sqrt{x} (\log{x})^2}{8\pi} + 1.12 \log{q}\log{x} ,
    \quad\text{for all}\quad
    x\geq 73.2 .
\end{equation}
Combine \eqref{eqn:principal_chi_case} and Theorem \ref{thm:psi_chi_expl_fmla} to recover Theorem \ref{thm:exact_formula}.

\subsection{Proof of Corollary \ref{cor:exact_formula}}\label{ssec:corollary1}

We require the following simple relationship:
\begin{align*}
    |\psi(x,\chi) - \theta(x,\chi)|
    &\leq \sum_{2\leq k\leq \left\lfloor\tfrac{\log{x}}{\log{2}}\right\rfloor} \theta(x^{\tfrac{1}{k}})
    < \frac{\log{x}}{\log{2}} \theta(\sqrt{x}),
    \qquad\text{where}\qquad
    \theta(x) = \sum_{p\leq x} \log{p} .
\end{align*}
Now, \cite[Cor.~2.1]{Broadbent} tells us that there is a constant $c_{\theta} = 1 + 1.93378\cdot 10^{-8}$ for all $x\geq 0$ such that $\theta(x) \leq c_{\theta} x$. It follows that
\begin{equation}\label{eqn:the_connection}
    \frac{|\psi(x,\chi) - \theta(x,\chi)|}{\sqrt{x}\log{x}} 
    < \frac{c_{\theta}}{\log{2}} 
    < 1.44270 .
\end{equation}
Therefore, Corollary \ref{cor:exact_formula} is a straightforward consequence of Theorem \ref{thm:exact_formula} and \eqref{eqn:the_connection}.

\section{The prime number theorem for primes in arithmetic progressions}\label{sec:corollary2}

In this section, we will prove our explicit versions of the prime number theorem for primes in arithmetic progressions (i.e. Corollaries \ref{cor:pntpaps}-\ref{cor:pntpaps_small_moduli}) using Theorem \ref{thm:exact_formula}. To this end, recall that for all $x\geq x_0\geq \max\{e^{10},q\}$ and any non-principal character $\chi$ modulo $q\geq 3$, we have 
\begin{equation*}
    |\psi(x,\chi)| \leq 
    \left(\frac{\log{x}}{8\pi} + \frac{\log{q}}{2\pi} + \Omega_0(x_0)\right) \sqrt{x}\log{x} + \Omega_1(x_0) \sqrt{x} + \Omega_2(x_0) .
\end{equation*}
Insert this observation into \eqref{eqn:convenient} to see that for all $x\geq x_0 \geq \max\{e^{10},q\}$, we have
\begin{align}
    &\left|\psi(x;q,a) - \frac{x}{\varphi(q)}\right|
    \leq \frac{1}{\varphi(q)} \left(\frac{\sqrt{x} (\log{x})^2}{8\pi} + 1.12 (\log{x})^2 + \sum_{\chi\neq\chi_0} |\psi(x,\chi)| \right) \nonumber\\
    &< \left(\frac{\log{x}}{8\pi\varphi(q)} + \left(1 - \frac{1}{\varphi(q)}\right)\frac{\log{q}}{2\pi} + \Omega_0(x_0) + \frac{\max\{\Omega_1(x_0),0\}}{\log{x_0}} + \frac{0.56\log{x_0}}{\sqrt{x_0}}\right) \sqrt{x}\log{x} + \Omega_2(x_0) \nonumber\\
    &:=  \left(\frac{\log{x}}{8\pi\varphi(q)} + \left(1 - \frac{1}{\varphi(q)}\right)\frac{\log{q}}{2\pi} + \Omega_3(x_0)\right) \sqrt{x}\log{x} + \Omega_2(x_0) , \label{eqn:psi_xqa}
\end{align}
since there are $\varphi(q)$ Dirichlet characters modulo $q$ and $\varphi(q)\geq 2$. Use \cite[Cor.~2.1]{Broadbent} in a similar way to how we used it in Section \ref{ssec:corollary1} to see
\begin{equation*}
    |\psi(x;q,a) - \theta(x;q,a)|
    \leq 1.44270 \sqrt{x}\log{x} .
\end{equation*}
Now, combine the preceding observations to see that for all $x\geq x_0 \geq \max\{e^{10},q\}$, we have
\begin{align}
    \left|\theta(x;q,a) - \frac{x}{\varphi(q)}\right|
    &< \left(\frac{\log{x}}{8\pi\varphi(q)} + \left(1 - \frac{1}{\varphi(q)}\right)\frac{\log{q}}{2\pi} + \Omega_3(x_0) + 1.44270\right) \sqrt{x}\log{x} + \Omega_2(x_0) \nonumber\\
    &:= \left(\frac{\log{x}}{8\pi\varphi(q)} + \left(1 - \frac{1}{\varphi(q)}\right)\frac{\log{q}}{2\pi} + \Omega_4(x_0)\right) \sqrt{x}\log{x} + \Omega_2(x_0) \label{eqn:theta_xqa}
\end{align} 
Finally, partial summation yields the relationship
\begin{equation}\label{eqn:p_summ}
    \pi(x;q,a) = \frac{\theta(x;q,a)}{\log{x}} + \int_2^x \frac{\theta(t;q,a)}{t(\log{t})^2}\,dt .
\end{equation}
Before we insert the asymptotic \eqref{eqn:theta_xqa} into \eqref{eqn:p_summ}, note that
\begin{equation*}
    \frac{x}{\log{x}} + \int_2^x \frac{dt}{(\log{t})^2} 
    = \Li(x) + \frac{2}{\log{2}}
    \quad\text{and}\quad
    \int_2^x \frac{x^{-\frac{1}{2}} dt}{\sqrt{t}\log{t}} 
    = \frac{\Ei\!\left(\frac{\log{x}}{2}\right) - \Ei\!\left(\frac{\log{2}}{2}\right)}{\sqrt{x}} ;
\end{equation*}
the latter decreases for all $\log{x} \geq 10$. It follows that for all $x\geq x_0\geq \max\{e^{10},q\}$, we have
\begin{align}
    \left|\pi(x;q,a) - \frac{\Li(x)}{\varphi(q)}\right|
    &< \frac{\sqrt{x}\log{x} + \int_2^x \frac{dt}{\sqrt{t}}}{8\pi\varphi(q)} \nonumber\\ 
    &\qquad + \left(\left(1 - \frac{1}{\varphi(q)}\right)\frac{\log{q}}{2\pi} + \Omega_4(x_0)\right) \left(1 + \frac{\Ei\!\left(\frac{\log{x_0}}{2}\right) - \Ei\!\left(\frac{\log{2}}{2}\right)}{\sqrt{x_0}} \right) \sqrt{x} \nonumber\\
    &\qquad+ \Omega_2(x_0) \left(\frac{1}{\log{x}} + \int_2^x \frac{dt}{t(\log{t})^2} \right) + \frac{1}{\log{2}} \nonumber\\
    &\leq \left(\frac{\log{x}}{8\pi\varphi(q)} + \frac{\Omega_5(x_0)\log{q}}{2\pi} + \Omega_6(x_0)\right)\sqrt{x} + \Omega_7(x_0) , \label{eqn:pi_xqa}
\end{align} 
in which
\begin{align*}
    \Omega_5(x_0) 
    = 1 + \frac{\Ei\!\left(\frac{\log{x_0}}{2}\right) - \Ei\!\left(\frac{\log{2}}{2}\right)}{\sqrt{x_0}}, \quad
    \Omega_6(x_0) 
    &= \frac{1}{8\pi} + \Omega_4(x_0)\Omega_5(x_0), \quad\text{and}\\
    \Omega_7(x_0) &= \frac{1 + \Omega_2(x_0)}{\log{2}} .
\end{align*}
Admissible computations for each of the $\Omega_i(x_0)$ that appear in  \eqref{eqn:psi_xqa}, \eqref{eqn:theta_xqa}, and \eqref{eqn:pi_xqa} at certain choices of $x_0$ are presented in Table \ref{tab:omega_values}. These computations complete Corollary \ref{cor:pntpaps}, by taking
\begin{equation*}
    a_1 = \Omega_5(x_0), \quad
    a_2 = \Omega_6(x_0), \quad
    a_3 = \Omega_7(x_0), \quad
    a_4 = \Omega_4(x_0), \quad
    a_5 = \Omega_2(x_0), \quad
    a_6 = \Omega_3(x_0).
\end{equation*}

\begin{table}[]
    \centering
    \begin{tabular}{c|cccccc}
        $\log{x_0}$ & $\Omega_2(x_0)$ & $\Omega_3(x_0)$ & $\Omega_4(x_0)$ & $\Omega_5(x_0)$ & $\Omega_6(x_0)$ & $\Omega_7(x_0)$ \\
        \hline
        $10$ & $-8.31179\cdot 10^{3}$ & $7.84909$ & $9.29179$ & $1.27146$ & $11.85396$ & $-1.19899\cdot 10^{4}$ \\
        $20$ & $-3.15402\cdot 10^{6}$ & $4.89427$ & $6.33697$ & $1.11315$ & $7.0938$ & $-4.55028\cdot 10^{6}$ \\
        $30$ & $-8.10901\cdot 10^{8}$ & $4.22563$ & $5.66833$ & $1.07187$ & $6.11552$ & $-1.16988\cdot 10^{9}$ \\
        $40$ & $-1.64282\cdot 10^{11}$ & $3.87641$ & $5.31911$ & $1.0528$ & $5.63974$ & $-2.37009\cdot 10^{11}$ \\
        $50$ & $-3.10314\cdot 10^{13}$ & $3.67311$ & $5.11581$ & $1.04175$ & $5.36916$ & $-4.47689\cdot 10^{13}$ \\
        $60$ & $-5.59133\cdot 10^{15}$ & $3.52631$ & $4.96901$ & $1.03453$ & $5.18036$ & $-8.06659\cdot 10^{15}$ \\
        $70$ & $-9.74775\cdot 10^{17}$ & $3.41787$ & $4.86057$ & $1.02944$ & $5.04344$ & $-1.4063\cdot 10^{18}$ \\
        $80$ & $-1.63943\cdot 10^{20}$ & $3.33388$ & $4.77658$ & $1.02566$ & $4.93893$ & $-2.36519\cdot 10^{20}$ \\
        $90$ & $-2.7664\cdot 10^{22}$ & $3.25716$ & $4.69986$ & $1.02274$ & $4.84652$ & $-3.99108\cdot 10^{22}$ \\
        $100$ & $-4.58854\cdot 10^{24}$ & $2.67049$ & $4.11319$ & $1.02042$ & $4.23696$ & $-6.61986\cdot 10^{24}$ \\
        $150$ & $-5.00022\cdot 10^{35}$ & $0.07749$ & $1.52019$ & $1.01352$ & $1.58052$ & $-7.2138\cdot 10^{35}$ \\
        $200$ & $-4.77543\cdot 10^{46}$ & $-1.33174$ & $0.11096$ & $1.0101$ & $0.15187$ & $-6.88948\cdot 10^{46}$ \\
        $250$ & $-4.4118\cdot 10^{57}$ & $-2.96611$ & $-1.52341$ & $1.00807$ & $-1.49591$ & $-6.36488\cdot 10^{57}$ \\
        $500$ & $-1.66057\cdot 10^{112}$ & $-12.17243$ & $-10.72973$ & $1.00402$ & $-10.73303$ & $-2.39569\cdot 10^{112}$
    \end{tabular}
    \caption{Computations for $\Omega_i(x_0)$ at several choices of $x_0$ in \eqref{eqn:psi_xqa}, \eqref{eqn:theta_xqa}, and \eqref{eqn:pi_xqa}.}
    \label{tab:omega_values}
\end{table}

To prove Corollary \ref{cor:pntpaps_small_moduli}, replace any occurrence of $\Omega_0(x)$, $\Omega_1(x)$, $\Omega_2(x)$ from Theorem \ref{thm:psi_chi_expl_fmla} in the preceding definitions with $\widetilde{\Omega_0}(x)$, $\widetilde{\Omega_1}(x)$, $\widetilde{\Omega_2}(x)$ from Theorem \ref{thm:psi_chi_expl_fmla_small_moduli}. 

\appendix

\section{The Chebotar\"{e}v density theorem}\label{sec:cdt}

Let $\mathbb{K}\subseteq \mathbb{L}$ be a Galois extension of number fields such that $n_{\mathbb{L}} = [\mathbb{L} : \mathbb{Q}]$, $\Delta_{\mathbb{L}}$ be the discriminant of $\mathbb{L}$, $\mathfrak{P}$ be a prime ideal of $\L$ that lies above a non-ramified prime ideal $\mathfrak{p}$ of the ring of integers $\mathcal{O}_{\mathbb{K}}$, and $\sigma_{\mathfrak{p}}$ be the Artin symbol associated to $\mathfrak{p}$; this is the conjugacy class of the Frobenius automorphism corresponding to $\mathfrak{P}/\mathfrak{p}$. If $\mathcal{C}$ is any conjugacy class in $\mathcal{G} = \Gal(\L/\K)$, then we let
\begin{equation*}
    \pi_{\mathcal{C}}(x) = \#\{\mathfrak{p}\subset\mathcal{O}_{\mathbb{K}} : \text{$\mathfrak{p}$ are prime ideals that do \textit{not} ramify in $\mathbb{L}$ and satisfy $\sigma_{\mathfrak{p}} = \mathcal{C}$}\}.
\end{equation*}
The celebrated Chebotar\"{e}v density theorem tells us that
\begin{equation}\label{eqn:CDT}
    \pi_{\mathcal{C}}(x) \sim \frac{\#\mathcal{C}}{\#\mathcal{G}} \Li(x) := \frac{\#\mathcal{C}}{\#\mathcal{G}} \int_2^x \frac{dt}{\log{t}} .
\end{equation}
Insert $\mathbb{L} = \mathbb{Q}(\omega_q)$ and $\mathbb{K} = \mathbb{Q}$, where $\omega_q$ is the $q^\text{th}$ root of unity, into \eqref{eqn:CDT} to retrieve the prime number theorem for primes in arithmetic progressions. In particular, under these choices, we have $n_{\mathbb{L}} = \#\mathcal{G} = \varphi(q)$ and $\#\mathcal{C} = 1$ for each conjugacy class $\mathcal{C}\subset\mathcal{G}$, so the prime number theorem for primes in arithmetic progressions tells us $\pi(x;q,a) \sim \Li(x)/\varphi(q)$ and is a special case of the Chebotar\"{e}v density theorem. Assuming the generalised Riemann hypothesis for the Dedekind zeta-function, Greni\'{e} and Molteni proved the latest explicit version of the Chebtar\"{e}v density theorem in \cite[Cor.~1]{Grenie2019}; their results improve a result from Osterl\'{e} \cite{Osterle}. In the end, for all $x\geq 2$, they prove
\begin{equation}\label{eqn:GM_CDT}
    \left|\pi_{\mathcal{C}}(x) - \frac{\#\mathcal{C}}{\#\mathcal{G}} \Li(x)\right|
    \leq \frac{\#\mathcal{C}}{\#\mathcal{G}} \sqrt{x} \left[\left(\frac{\log{x}}{8\pi} + \frac{1}{4\pi} + \frac{6}{\log{x}}\right) n_{\mathbb{L}} + \left(\frac{1}{2\pi} + \frac{3}{\log{x}}\right) \log{|\Delta_{\mathbb{L}}|}\right] .
\end{equation}
So, if $\mathbb{L}/\mathbb{K} = \mathbb{Q}(\omega_q)/\mathbb{Q}$, then $\Delta_{\mathbb{L}} = (-1)^{\tfrac{\varphi(q)}{2}} q^{\varphi(q)} \prod_{p|q} p^{-\tfrac{\varphi(q)}{p-1}}$ (see \cite[Prop.~2.7]{Washington}) and \eqref{eqn:GM_CDT} becomes
\begin{equation}\label{eqn:GM_CDT_sc}
\begin{split}
    \left|\pi(x;q,a) - \frac{\Li(x)}{\varphi(q)} \right|
    &\leq \left(\frac{\log{x}}{8\pi} + \left(1 + \frac{3}{\log{x}}\right)\frac{\log{q}}{2\pi} + \frac{1}{4\pi} + \frac{6}{\log{x}}\right) \sqrt{x} \\
    &\qquad\qquad\qquad\qquad\qquad\qquad- \sqrt{x} \left(\frac{1}{2\pi} + \frac{3}{\log{x}}\right) \sum_{p|q} \frac{\log{p}}{p-1} .
\end{split}
\end{equation} 

\bibliographystyle{amsplain}
\bibliography{refs}

\end{document}